\documentclass[final]{siamltex}%
\usepackage[dvips]{graphicx}
\usepackage{mathrsfs}
\usepackage{amsfonts}
\usepackage{subfigure,url}
\usepackage{color}
\usepackage{epsfig,overpic}
\usepackage{amsmath,latexsym}
\newcommand{\lbl}[1]{\label{#1}}

\def\eps{\varepsilon}

\def\d{\mathrm{d}}
\def\mbf{\mathbf}
\def\div{\mathrm{div}}

\newcommand\aint{{\int\!\!\!\!\!\!-}}

\title{Modified Wenzel and Cassie equations for wetting on rough surfaces\footnotemark[1]}%
\author{Xianmin Xu\footnotemark[2]
        }

\begin{document}
\maketitle

\renewcommand{\thefootnote}{\fnsymbol{footnote}}

\footnotetext[1]{This publication was based on
work supported  by Chinese NSFC project 11571354 and by SRF for ROCS, SEM.}
\footnotetext[2]{LSEC,
Institute of Computational Mathematics and Scientific/Engineering Computing, AMSS, NCMIS, Chinese Academy of Sciences, Beijing 100190, China.
({\tt xmxu@lsec.cc.ac.cn}).}
%\footnotetext[3]{Department of Mathematics, the Hong Kong University of Science and Technology, Clear Water Bay, Kowloon, Hong Kong, China. ({\tt mawang@ust.hk})}

\renewcommand{\thefootnote}{\arabic{footnote}}
\begin{abstract}
We study a stationary wetting problem on rough and inhomogeneous solid surfaces. 
We derive a new formula for the apparent contact angle by asymptotic two-scale homogenization method. 
The formula reduces to a modified Wenzel equation for  geometrically rough surfaces and
a modified Cassie equation for chemically inhomogeneous surfaces. Unlike the classical Wenzel and
Cassie equations, the modified equations correspond to local minimizers of the total interface energy in the solid-liquid-air system, so that they are consistent with experimental observations.  
The homogenization results are proved rigorously by a variational method.
\end{abstract}



\pagestyle{myheadings} \thispagestyle{plain} \markboth{X. XU}
 {Modified Wenzel and Cassie equations for wetting on rough surfaces}
\section{Introduction}
Wetting describes how liquid states and spreads on solid surfaces. It is a common phenomenon in nature and our daily life. It appears 
also in many industrial processes, such as painting, printing, chemical and petroleum industry, etc. 
Recently, wetting has arisen much interests in physics,
chemistry and material sciences(see \cite{Gennes85,Gennes03,Bonn09,Quere08} among many others).

The wetting property of a solid surface is mainly characterized by a contact angle, i.e. the angle between the liquid-air interface and the solid surface. On a planar and chemically
homogeneous surface, the static contact angle $\theta_Y$ is  given by  Young's equation\cite{Young1805}:
\begin{equation*}
\gamma_{LG}\cos\theta_Y=\gamma_{SG}-\gamma_{SL},
\end{equation*}
where $\gamma_{LG}$, $\gamma_{SG}$ and $\gamma_{SL}$ are surface tensions 
of the liquid-air, solid-air and solid-liquid interfaces, respectively.

One key problem in wetting is to study  how the roughness and
chemical inhomogeniety of a solid surface affect the apparent contact angle,
namely the macroscopic angle between the liquid-air interface and 
the solid surface. To answer this question, there are two very well-known equations, 
 the Wenzel
equation\cite{Wenzel36} and the Cassie equation\cite{Cassie44}.
The Wenzel equation characterizes the apparent contact angle $\theta_W$ on a geometrically 
rough surface(see the left subfigure in Fig.~\ref{fig:1}): 
\begin{equation}
\cos\theta_W=r\cos\theta_Y,
\end{equation}
where $\theta_Y$ is  Young's angle, and $r$ is a roughness factor, which is the
ratio of the area of the rough surface to the  area of the effective planar surface. 
On the other hand, the Cassie equation characterizes the apparent contact angle $\theta_C$ on a chemically patterned
surface(see the right subfigure in Fig.~\ref{fig:1}):
\begin{equation}
\cos\theta_C=\lambda\cos\theta_{Y1}+(1-\lambda)\cos\theta_{Y2},
\end{equation}
where $\theta_{Y1}$ and $\theta_{Y2}$ are respectively  Young's angles of
the two materials, $\lambda$ is the area fraction of material 1 on the surface.
Although the two equations are widely used, their validity
has been challenged recently\cite{Gao07,Bormashenko09,Marmur09,erbil2014debate}. The main problem is that the two equations are seldom consistent with experimental 
observations that the apparent contact angle is not unique on a rough surface. 
This phenomenon is very important in industrial applications and often referred to contact angle hysteresis(CAH).
Obviously, the classical Wenzel and Cassie equations can not describe CAH phenomena \cite{Johnson1964,Extrand02,Schwarts1985,XuWang2011}.

  \begin{figure}[htb!]
\vspace*{-2mm}
    \centering
  \resizebox{!}{3.cm}%{height=6.5cm,width=10cm}
    {\includegraphics[height=2cm]{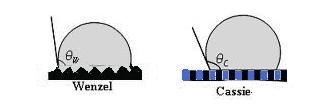}}
    \vspace*{-5mm}
    \caption{The contact angle on geometrically and chemically rough surfaces.}    \lbl{fig:1}
 \end{figure}

Mathematically, the classical Wenzel  and Cassie equations can be
proved to hold when considering the global minimizers of  the
total interface energy in a liquid-air-solid system\cite{Alberti05,Whyman08,xu-wang10,chenWangXu2013,caffarelli2007capillary1}. 
It is also known that the two equations
may fail when the system admits only local minimizers\cite{Gao07,XuWang2011},
that really happens in CAH phenomena. 
To characterize the CAH quantitatively, some modified forms of the Cassie equation have been 
developed for chemically patterned surfaces\cite{Choi09,Raj12,XuWang2013}. 
The modified formulas are found to be consistent with experiments very well\cite{Priest2013,Raj12}. 
The work has been extended to a special geometrically rough surface in \cite{Xu2016}.
However, these formulas  for local minimizers have not been proved rigorously. Furthermore, there 
is no result in literature for general geometrically and chemically rough surfaces.

%This explains why the Wenzel and Cassie equations do not hold in general.
%This has been proved mathematically for a chemically patterned surface in two dimension\cite{XuWang2011}.

 %It is believed
%that the modified Cassie equation can be used to describe the 
%general equilibrium state when might correspond to the local minimizer
%of the system\cite{Choi09,Raj12,XuWang2013,Priest2013}. However, this has now been proved rigorously.
%Furthermore, it is also interesting to ask how the classical Wenzel equation should be modified 
%for geometrically rough surface. This is not known in literature.

In this paper, we consider a wetting problem on a general rough surface, which
includes both geometrical  roughness and chemical  inhomogeneity. 
The static liquid-air interface is described by a minimal surface equation
coupled with a contact angle condition.
By the asymptotic two-scale homogenization method, we derive a new formula for the apparent contact angle 
of the liquid-air interface. The formula shows that both the geometric
and chemical properties must be averaged on a contact line to give the correct value for
an apparent contact angle. The formula can be reduced to  a
 modified Wenzel equation for  geometrically rough  surfaces and a modified Cassie equation
for chemically inhomogeneous surfaces. In comparison with the classical Wenzel and Cassie equations,
the two modified equations correspond to local minimizers of the system and thus are 
consistent with  experimental observations. Furthermore, we prove
the homogenization results rigorously by a variational approach.  More precisely,
we proved that the difference between the liquid-gas interface and
the homogenized surface is of order $O(\varepsilon)$, where $\varepsilon$ is
a small parameter of the roughness. The key idea of the proof is to construct
an auxiliary energy minimizing problem and to estimate the energy of the minimizer 
directly. We also discuss how to use the modified Wenzel and Cassie equations
 to understand  contact angle hysteresis phenomena.

%{\color{red} Although the new formula is derived and proved only for the periodic 
%surface inhomogeneity, the analysis can be generalized to general cases.
%Therefore, we think the formula also holds for other cases. The key
%issue is that the macroscopic contact angle depends only on the 
%properties near the contact line. However, to quantify the contact angle
%correctly, both the geometric information and chemical information must be 
%averaged in a proper way. We also use this formula to understand 
%the contact angel hysteresis by showing some examples.
%}

The structure of the paper is as follows. In section 2, we describe a simple model 
problem of wetting. In section 3, we do homogenization analysis for 
the model problem to derive a formula for the apparent contact angle. 
In section 4, we discuss the physical meaning of the formula in various 
situations to give the modified Wenzel and Cassie equations. In 
section 5, we prove the homogenization results rigorously. In section 6,
we give some general discussions on how to use the formula in practice. 
Finally, some conclusion remarks are given in section 7.

\section{The mathematical problem}
When a vertical solid wall is inserted into a liquid reservoir, the liquid-air interface will
rise or descend along the  wall due to the wetting property of the solid surface. 
Such a system can be used to study the wetting phenomena on rough surfaces\cite{Johnson1964}.
A simplified mathematical model for this problem is described as follows.

As in Fig.~\ref{fig:2}, let the  solid surface $S^\eps$ be given by
\begin{equation}
S^\eps:=\{ x=h_{\eps}(y,z)=\eps h(y/\eps,z/\eps) | y\in (0,a), z\in(-M,M)\}, \label{e:roughsurf}
\end{equation}
where $h(Y,Z)$ is a periodic function with respect to $Y$ and $Z$ with period 1,
$a>0$ is a constant, and $M>0$ is  large enough. Here  $\eps\ll1$
is a small parameter to characterize the roughness of the surface.
In addition, we also assume that the solid surface is chemically inhomogeneous in the sense that the (static) Young's angle on the surface is  given by 
$\theta_s(y/\eps,z/\eps)$ with $\theta_s(Y,Z)$ being period in  $Y$ and $Z$ with period 1.

Suppose the liquid-gas interface $\Gamma^{\eps}$  is given by
  \begin{equation}
 \Gamma^\eps:=\{ z=u_{\eps}(x,y) | x\in(0,1),y\in(0,a)\}, \label{e:liquidsurf}
  \end{equation}
  such that  $|u_{\eps}|<M$ and 
 \begin{equation}u_{\eps}(1,y)=0.\label{e:bnd}
  \end{equation}
 We  assume that $u_{\eps}(x,y)$ is periodic in $y$ with period $\eps$. 
 Furthermore, we denote the contact line, i.e. the intersection of $S^\eps$ and $\Gamma^\eps$, as
 \begin{equation}
 L^{\eps}: \qquad
  \left\{
 \begin{array}{l}
 x=\phi_\eps(y), \\
 z=\psi_\eps(y).
 \end{array}
 \right.
 \end{equation}
 It is easy to see that $\phi_\eps$ and $\psi_\eps$ are  determined by $h_\eps$ and $u_\eps$, and  
 periodic with respect to $y$ with period $\eps$.
 % By \eqref{e:roughsurf}, we can rewrite  $\phi_\eps$ as $\phi_\eps=\eps\phi(x/\eps)$ for some periodic function $\phi$.
  \begin{figure}[htb!]
\vspace*{-5mm}
    \centering
  \resizebox{!}{7.5cm}%{height=6.5cm,width=10cm}
    {\includegraphics[height=4cm]{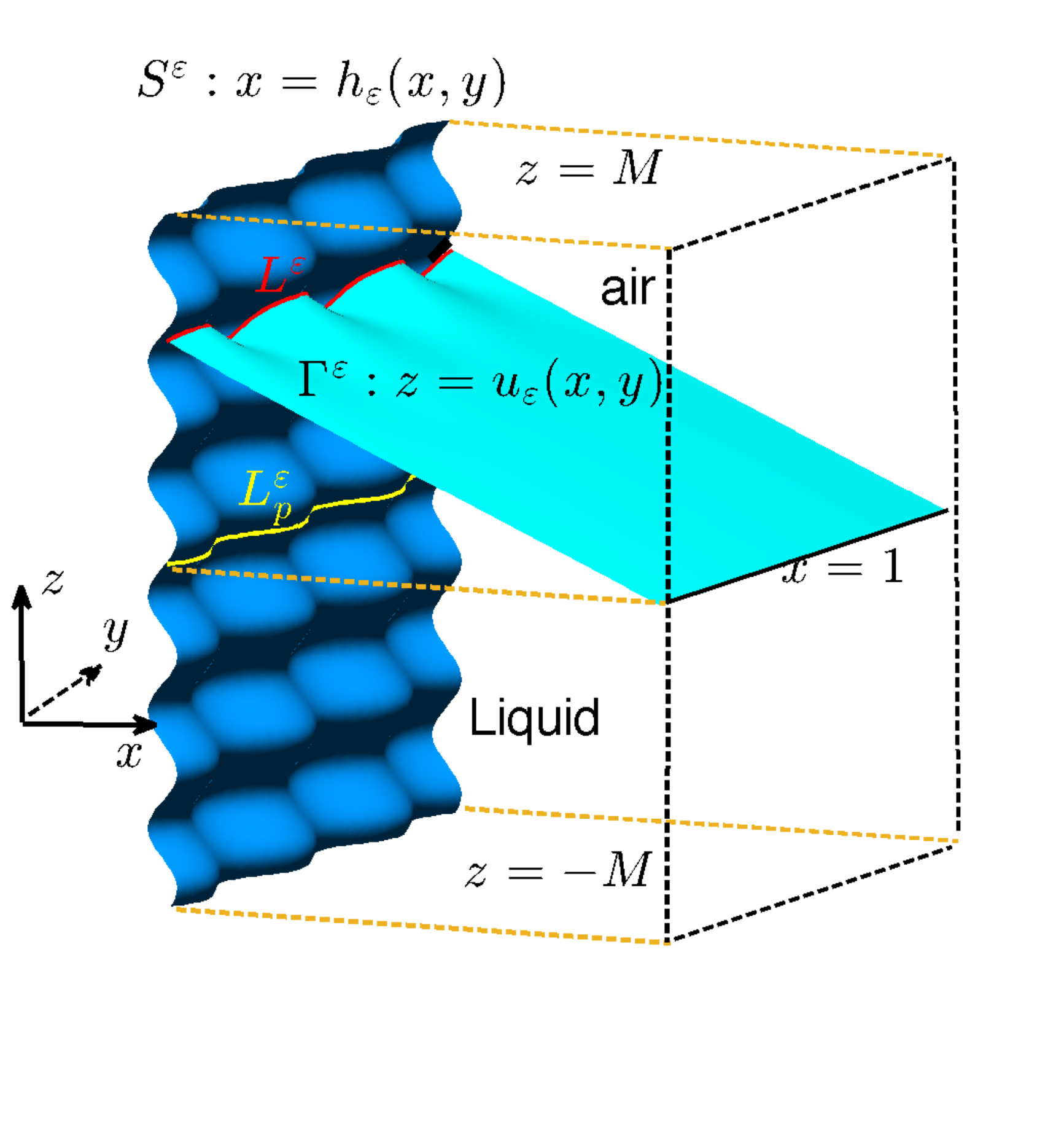}}
    \vspace*{-10mm}
    \caption{The rough surface and the fluid-air interface.}    \lbl{fig:2}
 \end{figure}
 
Now we derive the equation satisfied by the function $u_\eps$.  Since we are 
 interested only in the  contact angle between $\Gamma^\eps$ and $S^\eps$,
 we can ignore the gravity. By the Young-Laplace equation\cite{landaufluid}, in
 equilibrium, the mean curvature $\kappa$ of the interface $\Gamma^\eps$ is determined only by
 the pressure jump across the liquid-air interface, 
 \begin{equation}\label{e:YL}
 \kappa =\frac{[p]}{2\gamma},
 \end{equation}
where $ [p]$ is the pressure jump and $\gamma$ is the surface tension of the interface.
 On the contact line 
$L^\eps$, the microscopic contact angle is equal to the local Young's angle\cite{Young1805}, 
\begin{equation}\label{e:Young}
\mathbf n_{\Gamma}^\eps\cdot \mathbf n_{S}^\eps=\cos\theta_s^{\eps}(y)
\end{equation}
where $\theta_s^\eps(y)=\theta_s(\frac{y}{\eps},\psi_\eps(y))$, $\mathbf{n}^\eps_{\Gamma}$ and $\mathbf{n}^\eps_S$
are the unit normal vectors of the liquid-air interface and the solid surface, respectively,  defined
as,
\begin{equation}\label{e:normal}
\mbf n_\Gamma^\eps:=\frac{(-\partial_x u_\eps,-\partial_y u_\eps,1)^T}{\sqrt{1+(\partial_x u_{\eps})^2+(\partial_y u_{\eps})^2}},
\qquad
\mbf n_S^\eps:=\frac{(1, -\partial_y h_\eps,-\partial_z h_\eps)^T}{\sqrt{1+(\partial_y h_\eps)^2+(\partial_z h_\eps)^2}}.
\end{equation}

For simplicity, we assume the pressure jump $[p]=0$, then the equations~\eqref{e:YL},
\eqref{e:Young} and  the boundary condition \eqref{e:bnd} can be rewritten as
\begin{equation}\label{e:pb3}
\left\{
 \begin{array}{ll}
\div\Big(\frac{\nabla u_\eps}{\sqrt{1+|\nabla u_{\eps}|}}\Big)=0 ,& \hbox{ in } D^\eps \\
\mathbf{n}_{S}^\eps\cdot\mbf{n}_{\Gamma}^\eps=\cos\theta^\eps_s ,& \hbox{ on } L_p^{\eps},\\
u_\eps(1,y)=0, \\
 u_\eps(x,y) \hbox{ is periodic in $y$ with period $\eps$,}
 \end{array}
 \right.
\end{equation}
where $\nabla=(\partial_x,\partial_y)^T$, $\div:=\partial_x+\partial_y$, and 
\begin{equation}\label{e:domain}
D^\eps:=\{(x,y)|\phi_\eps(y)<x<1,0<y<a\}, \hbox{ and } L_p^{\eps}:=\{(x,y)| x=\phi_\eps(y),0<y<a\}.
\end{equation}
Here we use the standard formula for mean curvature $\kappa=\div\big(\frac{\nabla u_\eps}{\sqrt{1+|\nabla u_{\eps}|}}\big)$.

%In addition, $u_{\eps}$ also satisfies the boundary conditions:
%\begin{equation}
%u_\eps(1,y)=0, \qquad u_\eps(1,y) \hbox{ is periodic in $y$ with period $1$.}
%\end{equation}

When $\eps$ is small, $\Gamma^\eps$ is a perturbation of a homogenized surface $\Gamma^0$(see Fig.~\ref{fig:homo}):
\begin{equation}\label{e:gamma0}
\Gamma^0:=\{ z=u_0(x) | y\in(0,a),x\in(0,1)\},
\end{equation}
where $u_0$ is a function depending only in $x$. Denote $\theta_a$ the angle between $\Gamma^0$ 
and the macroscopic solid surface 
$$S^0:=\{x=0| y\in(0,a),z\in(-M,M)\}.$$
We are mainly
interested in how the macroscopic apparent contact angle $\theta_a$ depends on local geometric and chemical roughness 
near the contact line. %This will be discussed in the following sections.

In the end of the section, we would like to remark that, in general the equation~\eqref{e:pb3} may 
not admit a solution, since the associated minimal surface 
may not be $z$-graph. This happens in the Cassie-Baxter state of wetting, where air may be trapped
under liquid\cite{chenWangXu2013}. In this paper, we always assume \eqref{e:pb3} admits a solution.
%This should be true under some assumptions for $h$ and $\theta_s$.
  \begin{figure}[htb!]
\vspace*{-3mm}
    \centering
  \resizebox{!}{7.5cm}%{height=6.5cm,width=10cm}
    {\includegraphics[height=4cm]{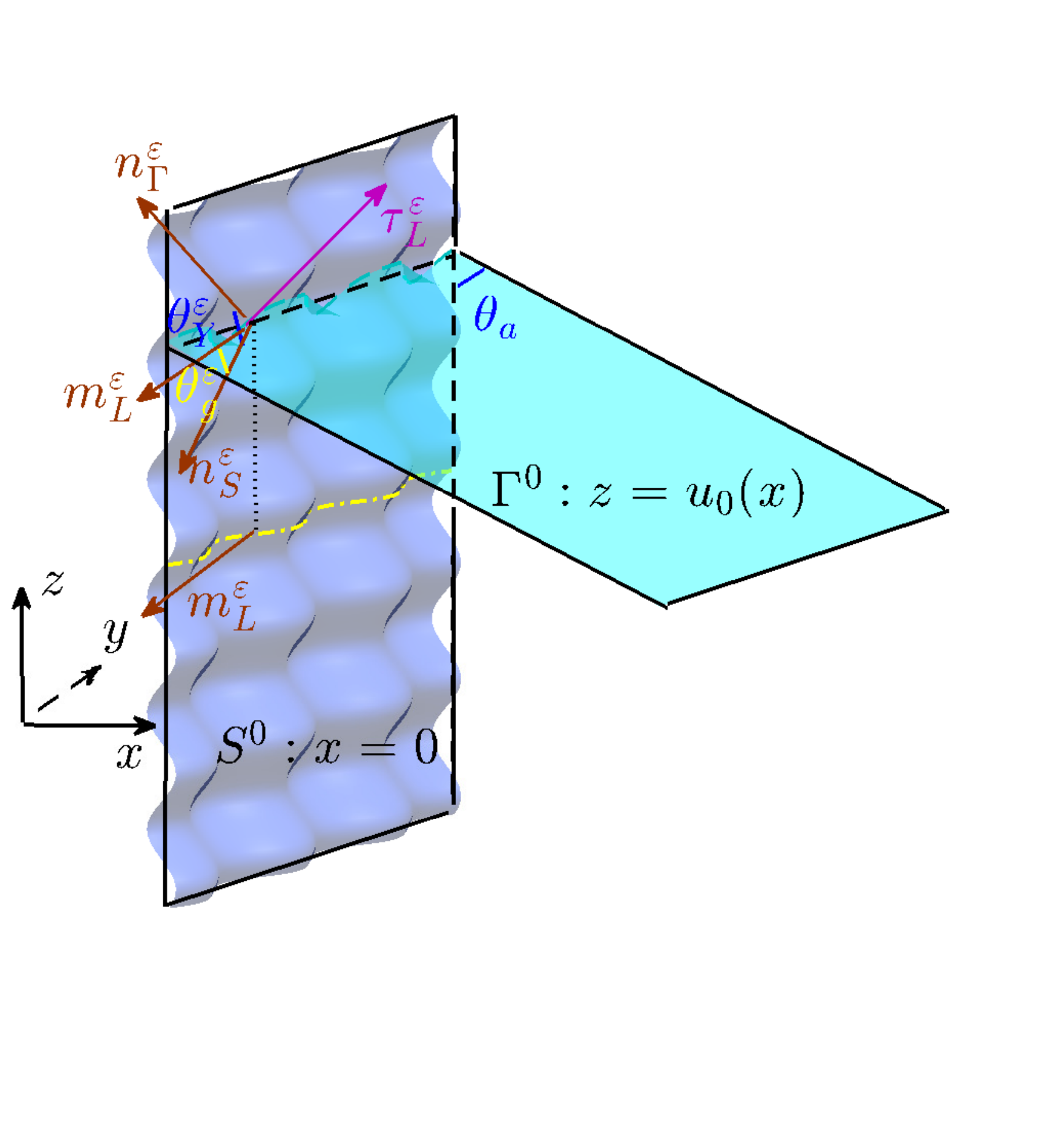}}
    \vspace*{-10mm}
    \caption{The homogenized interface and the apparent contact angle.}    \lbl{fig:homo}
 \end{figure}
 
 \section{Homogenization}
In this section, we  derive the formula for the homogenized interface $\Gamma^0$ and the 
apparent contact angle $\theta_a$ by asymptotic two-scale homogenization analysis. 
 Since $\eps$ is small, we can assume that 
 the solution $u_{\eps}$ of \eqref{e:pb3} is oscillating mainly in the vicinity of the contact line $L^\eps$. 
Far from $L^\eps$, we  can assume that the interface $\Gamma^\eps$ is a perturbation
 of a macroscopic surface $\Gamma^0$.
We can do outer and inner expansions of $u_{\eps}$ as follows. 

{\it Outer expansion.} Far from the contact line $L^\eps$, we let 
\begin{equation}\label{e:outexp}
u_{\eps}(x,y)=u_0(x)+\eps u_1(x,\frac{y}{\eps})+\eps^2 u_2(x,\frac{y}{\eps})+\cdots.
\end{equation}
Denote the fast parameter $Y=\frac{y}{\eps}$, then $u_i(x,Y)$ is periodic in $Y$ with period 1
for $i\geq 1$, and $\partial_y=\eps^{-1}\partial_Y$.

 Substitute the above expansion to the first equation of  \eqref{e:pb3}. 
The leading order in $\eps$ of the equation is
\begin{align}
\partial_Y\Big(\frac{\partial_Y u_1}{\sqrt{1+(\partial_x u_0)^2+(\partial_Y u_1)^2}}\Big)=0.
\end{align}
This leads to 
\begin{align*}
\frac{\partial_Y u_1}{\sqrt{1+(\partial_x u_0)^2+(\partial_Y u_1)^2}}=c_0(x),
\end{align*}
for some function $c_0(x)$ depending only in $x$ and such that $|c_0(x)|<1$.
Direct calculations give
\[
(\partial_Y u_1)^2=\frac{c_0^2(1+(\partial_x u_0)^2)}{1-c_0^2}=:c_1(x).
\]
This leads to 
$$u_1=\pm\sqrt{c_1(x)}Y+c_2(x),$$
for  a function $c_2$ independent of $Y$. Noticing that $u_1$ is periodic in $Y$ with period $1$,
we have $c_1(x)=0$ and $u_1$ depends only in $x$.

The next order of the expansion of the first equation of \eqref{e:pb3} gives
\begin{equation}
\partial_x\Big(\frac{\partial_x u_0}{\sqrt{1+(\partial_x u_0)^2}}\Big)
+
\partial_Y\Big(\frac{\partial_Y u_2}{\sqrt{1+(\partial_x u_0)^2}}\Big)=0.
\end{equation}
Integral the equation with respect to $Y$ in the interval $(0,1)$. Noticing that $u_2$ 
is period in $Y$ with period 1, we have
\begin{equation}\label{e:out1}
\partial_x\Big(\frac{\partial_x u_0}{\sqrt{1+(\partial_x u_0)^2}}\Big)=0.
\end{equation}
%Or equivalently
%\begin{equation}
%\frac{\partial_{xx} u_0}{(\sqrt{1+(\partial_x u_0)^2})^3}=0.
%\end{equation}
This  implies
 \begin{equation}\label{e:2ndderiv_u0}
 \partial_{xx}u_0=0,
 \end{equation}
i.e. $u_0(x)$ is a linear function of $x$.

Now we consider the boundary conditions. Substitute the expansion~\eqref{e:outexp} to the third equation of \eqref{e:pb3}.
The leading order implies that $u_0(1)=0$. Together with \eqref{e:2ndderiv_u0}, we have
 \begin{equation}\label{e:linearu0}
u_0=k(1-x),
 \end{equation}
for some constant $k$. 
The equation implies  that the homogenized liquid-air interface $\Gamma^0$ is  planar and has an apparent contact angle(with
respect to $S^0$) given by
\begin{align}
\cos\theta_a=\frac{k}{\sqrt{1+k^2}}.\lbl{e:out3}
\end{align}

{\it Inner expansions.} Now we consider the expansion of $u_\eps$ 
near the contact line $L^\eps$.
We assume that the interface oscillates around $\{z=\hat{u}_0\}$ near the rough solid surface:
\begin{align}
u_{\eps}(x,y)=\hat{u}_0+\eps \hat{u}_1(X,Y)+\eps^2\hat{u}_2(X,Y)+\cdots,\lbl{e:inEx}
\end{align}
where $X=\frac{x}{\eps}, Y=\frac{y}{\eps}$ and $\hat{u}_i(X,Y)$ is periodic in $Y$ with period $1$, $i\geq 1$.
Since the contact line $L^\eps=\{(x,y,z): x=\phi_\eps(y),z=\psi_\eps(y)\}$  is the intersection of $\Gamma^\eps$ and $S^\eps$,  we also have 
\begin{equation}
\label{e:exp_phi}
\phi_\eps(y)=\eps\hat \phi_1(Y)+\eps^2\hat \phi_2(Y)+\cdots,
\end{equation}
and 
$$\psi_\eps(y)=\hat{u}_0+\eps\hat{\psi}_1(Y)+\cdots.$$
%For convenience, we also  introduce the fast parameter in $z$-direction:
%$$\qquad Z=\frac{z-\hat{u}_0}{\eps}.$$ 

 Direct calculations lead to
\begin{equation*}
\nabla =\eps^{-1} \nabla_{\mbf X},
\end{equation*}
with $\nabla_{\mbf X}=(\partial_X,\partial_Y)$
, and 
\begin{equation}\label{e:exp_nGamma}
\mathbf{n}_{\Gamma}^\eps=\mathbf{n}_{\Gamma}^0+\eps\mathbf{n}_{\Gamma}^1+\cdots,
\end{equation} 
with 
\begin{equation}\label{e:nGamma0}
\mathbf{n}_{\Gamma}^0=\frac{(-\partial_X u_1,-\partial_Y u_1,1)}{\sqrt{1+|\nabla_{\bf X} u_1|^2}}.
\end{equation}
%In addition, we have 
%\begin{equation}\label{e:nS0}
%\mathbf{n}_{S}^\eps=\frac{(1,\partial_Y h, \partial_Z h)}{\sqrt{1+(\partial_Y h)^2+(\partial_Z h)^2}}.
%\end{equation}
We submit the expansion~\eqref{e:inEx} to the first equation of~\eqref{e:pb3}. On the leading order, we have 
\begin{align}
&\nabla_{\mbf X}\cdot\left(\frac{\nabla_{\mbf X}\hat u_1}{\sqrt{1+(\nabla_{\mbf X}\hat{u}_1)^2}}\right)=0, 
\qquad\qquad  X>\hat \phi_1(Y), \lbl{e:in1}
%\\
%&\mathbf{n}_{\Gamma}^0 \cdot \mathbf{n}_{S}
%=\cos\hat{\theta}_s(Y), \qquad\qquad\qquad\qquad X=\hat \phi_1(Y).\lbl{e:in2}
\end{align}
%Here $\hat{\theta}_s(Y)=\theta_s(Y,\hat{u}_0+\eps \psi_1(Y))$.

%\subsection{The apparent contact angle}\\
{\it The apparent contact angle.}
We now derive how the apparent contact angle $\theta_a$ depends on the microscopic properties in the system.
We will need the matching condition between the inner and outer expansions:
 \begin{align}
\lim_{x\rightarrow 0} u_0=\hat{u}_0,
\qquad \lim_{x\rightarrow 0}\partial_x u_0=\lim_{X\rightarrow\infty}\partial_{X}\hat{u}_1,\qquad
\lim_{x\rightarrow 0}\partial_y u_0=\lim_{X\rightarrow\infty}\partial_{Y}\hat{u}_1.\lbl{e:match}
\end{align}

 Integral the equation \eqref{e:in1} in a domain 
 $$
 \hat{D}_{T}=\{(X,Y):  \hat \phi_1(Y)<X<T, 0<Y<1\},
 $$
 for any $T>0$.  Integration by part leads to
 \begin{align}
0&=\int_{\hat{D}_T}\nabla_{\mbf X}\cdot\left(\frac{\nabla_{\mbf X}\hat u_1}{\sqrt{1+(\nabla_{\mbf X}\hat{u}_1)^2}}\right)\d X \d Y\nonumber\\
%&=\int_{L_p}\hat{\mbf n}_{\Gamma} \d S\nonumber\\
&=-\int_{\hat{L}_{p}} \frac{\nabla_{\mbf X}\hat u_1}{\sqrt{1+(\nabla_{\mbf X}\hat{u}_1)^2}}
\cdot\frac{(1,-\partial_Y\hat \phi_1)^T}{\sqrt{1+(\partial_Y\hat \phi_1)^2}} \d S 
+\int_{0}^1\frac{\nabla_{\mbf X}\hat u_1}{\sqrt{1+(\nabla_{\mbf X}\hat{u}_1)^2}}|_{X=T}\cdot (1,0)^T \d Y,\lbl{e:tem}
 \end{align}
 where $\hat{L}_p:=\{(X,Y)|X=\hat \phi_1(Y),0<Y<1\}$.
 This further leads to
 \begin{align}\label{e:matchBC}
 \int_{0}^1\frac{\partial_X\hat u_1}{\sqrt{1+(\nabla_{\mbf X}\hat{u}_1)^2}}|_{X=T} \d Y=\int_{\hat{L}_{p}} \frac{\nabla_{\mbf X}\hat u_1}{\sqrt{1+(\nabla_{\mbf X}\hat{u}_1)^2}}
\cdot\frac{(1,-\partial_Y \hat \phi_1)^T}{\sqrt{1+(\partial_Y \hat \phi_1)^2}} \d S .
 \end{align}
 
Let $T$ goes to infinity. 
Using the matching condition \eqref{e:match}, the left hand side term of \eqref{e:matchBC} reduces  to,
\begin{align}
\lim_{T\rightarrow\infty} \int_{0}^1\frac{\partial_X\hat u_1}{\sqrt{1+(\nabla_{\mbf X}\hat{u}_1)^2}}|_{X=T} \d Y
=\frac{\partial_x u_0}{\sqrt{1+(\partial_x u_0)^2}}=
\frac{-k}{\sqrt{1+k^2}}=-\cos\theta_a.\label{e:left}
\end{align}
For the right hand side term, we get
\begin{align}
\int_{\hat{L}_{p}} \frac{\nabla_{\mbf X}\hat u_1}{\sqrt{1+(\nabla_{\mbf X}\hat{u}_1)^2}}
\cdot\frac{(1,-\partial_Y \hat \phi_1)^T}{\sqrt{1+(\partial_Y \hat \phi_1)^2}} \d S 
&=
\int_{\hat{L}_{p}} \frac{(\partial_X u_1, \partial_Y u_1,-1)}{\sqrt{1+(\nabla_{\mbf X}\hat{u}_1)^2}}
\cdot\frac{(1,-\partial_Y \hat \phi_1,0)^T}{\sqrt{1+(\partial_Y \hat \phi_1)^2}} \d S \nonumber\\
&= -\int_{\hat{L}_{p}} \mbf n_{\Gamma}^0\cdot\mbf m_{L}^0 \d S.\label{e:right}
\end{align}
Here $\mbf n_{\Gamma}^0$ is defined in~\eqref{e:nGamma0} and 
$\mbf m_{L}^0:=\frac{(1,-\partial_Y \hat \phi_1,0)^T}{\sqrt{1+(\partial_Y \hat \phi_1)^2}}$, which 
is the inner normal of $\hat{L}_p$. The above three equations implies
\begin{equation}\label{e:relation0}
\cos\theta_a=\frac{k}{\sqrt{1+k^2}}= \int_{\hat{L}_{p}} \mbf n_{\Gamma}^0\cdot\mbf m_{L}^0 \d S
=\int_0^1(\mbf n_{\Gamma}^0\cdot\mbf m_{L}^0)\sqrt{1+(\partial_Y\hat \phi_1)^2}\d Y .
\end{equation}

In reality, it is more convenient to use the integral on $L^\eps_p$ instead of on $\hat{L}^\eps_p$.
Denote 
\begin{equation}\label{e:normalCL}
 \mbf m_L^\eps=\frac{(1,-\partial_y\phi_\eps,0)}{\sqrt{1+(\partial_y \phi_\eps)^2}}
\end{equation}
the inner normal of $L^\eps_p$. It is easy to see that  
$$\mbf m_L^{\eps}=\mbf m_L^0+\eps \mbf m_L^1+\cdots.$$ 
By the equation~\eqref{e:exp_phi},\eqref{e:exp_nGamma} and the above expansion, we could use the following formula instead of \eqref{e:relation0}:
\begin{align}
\cos\theta_a&\approx 
\aint_0^{\eps}(\mbf n_{\Gamma}^\eps\cdot\mbf m_{L}^\eps)
%\big(\cos\theta_s(y/\eps,\psi_\eps(y))-\cos\theta_g(y/\eps,\psi_\eps(y))\big)
\sqrt{1+(\partial_y\phi_\eps)^2}\d y.\label{e:appAng0}
\end{align}
Hereinafter, we use the notation $\aint_0^{\eps}=\frac{1}{\eps}\int_0^\eps$.  
The inner product $\mbf n_{\Gamma}^\eps\cdot\mbf m_{L}^\eps$ in the equation can be
understood as follows. Denote  
$$\tau^\eps_L:=\frac{( \partial_y\phi_\eps,1,\partial_y\psi_\eps)^T}{\sqrt{1+(\partial_y\phi_\eps)^2+(\partial_y\psi_\eps)^2}},$$
  the unit tangential vector of the contact line $L^\eps$.
If we define an angle $\theta_g^\eps$ on $L^\eps$(see Fig.~\ref{fig:homo}),
\begin{equation}
\label{e:geoAng}
\theta_g^\eps(y):=\arcsin((\mbf m_L^\eps\times\mbf n^\eps_S)\cdot\tau_L^\eps),
\end{equation}
that depends only in the geometric property of the rough surface on the contact line. Noticing 
$\mathbf{n}^\eps_\Gamma\cdot\mathbf{n}^\eps_{S}=\cos\theta_s^\eps(y)$ on $L^\eps$(see in \eqref{e:pb3}), and the geometric relation  $\tau_L^\eps\cdot\mathbf{m}_L^\eps=0, \tau_L^\eps\cdot\mbf n_{\Gamma}^\eps=0$, and $\tau_L^\eps\cdot\mbf n_S^\eps=0$,
we easily have 
\begin{equation}\label{e:innproduct}
\mathbf n_{\Gamma}^\eps\cdot\mbf m_{L}^\eps=\cos(\theta_s^\eps(y)- \theta_g^\eps(y)).
\end{equation}
%The equation~\eqref{e:appAng0} can be rewritten as
%\begin{align}
%\cos\theta_a&= 
%\aint_0^{\eps}\cos\big(\theta^\eps_s(y)-\theta_g^\eps(y)\big)\sqrt{1+(\partial_y\phi_\eps)^2}\d y.%\label{e:appAng1}
%\end{align}

%with  $\tau$ being the tangential direction of the contact line, $\mbf m_L$ is the normal of $L_p^\eps$, the projection of
%the contact line $L^\eps$ in $z=0$ surface. They are given by 
%\begin{align*}
%&\tau_{L}=\frac{( \partial_y\phi_\eps,1,\partial_y\psi_\eps)^T}{\sqrt{1+(\partial_y\phi_\eps)^2+(\partial_y\psi_\eps)^2}},\\
%&\mbf m_L=\frac{(1,\partial_y\phi_\eps,0)}{\sqrt{1+(\partial_y \phi_\eps)^2}}
%\end{align*}

Combing the equations~(\ref{e:linearu0})-(\ref{e:out3}) and (\ref{e:appAng0})-(\ref{e:innproduct}), we know that $\Gamma^0$ is  a planar surface:
\begin{equation}\label{e:homo}
z=u_0(x):=k(1-x),
\end{equation}
and the apparent contact angle $\theta_a$ is given by 
\begin{equation}\label{e:ModifiedWC}
\cos\theta_a=\frac{k}{\sqrt{1+k^2}}=
\aint_0^{\eps}\cos\big(\theta^\eps_s(y)-\theta_g^\eps(y)\big)\sqrt{1+(\partial_y\phi_\eps)^2}\d y.
\end{equation}
The formula \eqref{e:ModifiedWC} characterizes how the macroscopic contact angle
depends on the local chemical and geometric information of the rough surface on the contact line. 

\section{The modified Wenzel equation and the modifed Cassie equation}
 In this section, we will
describe the  physical meanings  of \eqref{e:ModifiedWC}.
We can see that the equation gives some new results which are 
different from the classical  Wenzel and Cassie equations.
Specifically, we will derive a modified Cassie equation
for chemically patterned surface and a modified Wenzel equation
for geometrically rough surface.

{\it Young's angle.}
Firstly, we show that when the solid surface is flat and homogeneous(see Fig.~\ref{fig:planar}(a)),
the equation \eqref{e:ModifiedWC}
will give a Young's angle.
Actually, when the surface is flat,  the solid boundary is given by $\{x=0\}$. The normal vector $\mathbf{m}_L$
is parallel to $\mathbf{n}_{S}$. Then we have 
$\theta_g=0.$
Since the surface is homogeneous,  the static contact angle function $\theta_s^{\eps}(y)$ is equal to a constant $\theta_Y$. 
Furthermore, the contact line is in the solid surface so that $\phi_\eps=0$.
So the equation~\eqref{e:ModifiedWC} is reduced   to
\begin{align}
\cos\theta_a=\cos\theta_Y.
\end{align}
This implies $\theta_a=\theta_Y$, i.e. the macroscopic contact angle is equal to  Young's angle of the solid surface. This is consistent with the  physical 
principle that,  on a planar homogeneous surface,  the contact angle is  Young's angle.
  \begin{figure}[htb!]
\vspace*{-2mm}
    \centering
      \subfigure[homogeneous planar solid surface]{
      \includegraphics[height=5.5cm]{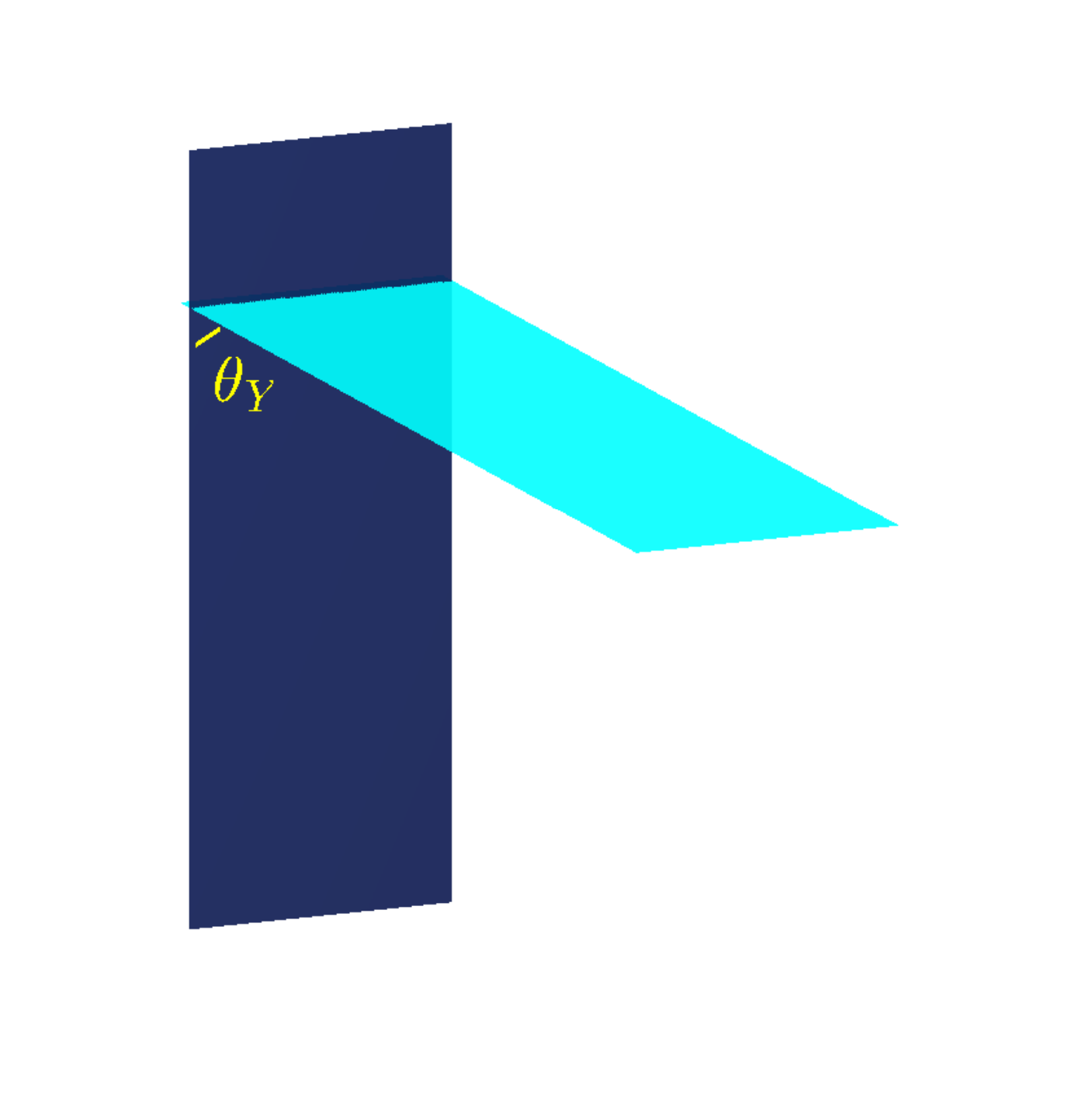}
      }
       \subfigure[inhomogeneous planar solid surface]{
      \includegraphics[height=5.5cm]{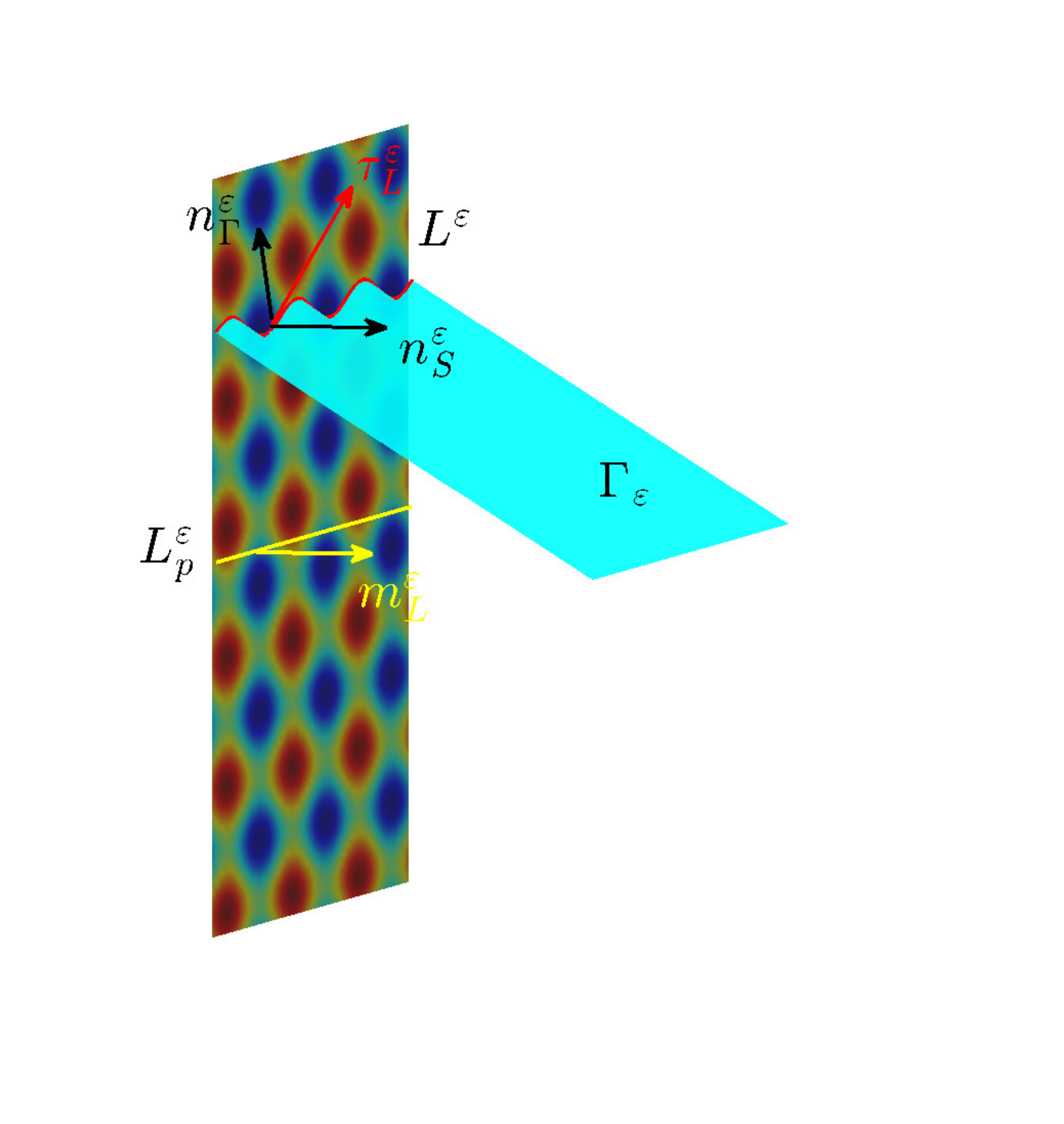}
      }
        \subfigure[inhomogeneous planar solid surface with striped pattern]{
      \includegraphics[height=5.5cm]{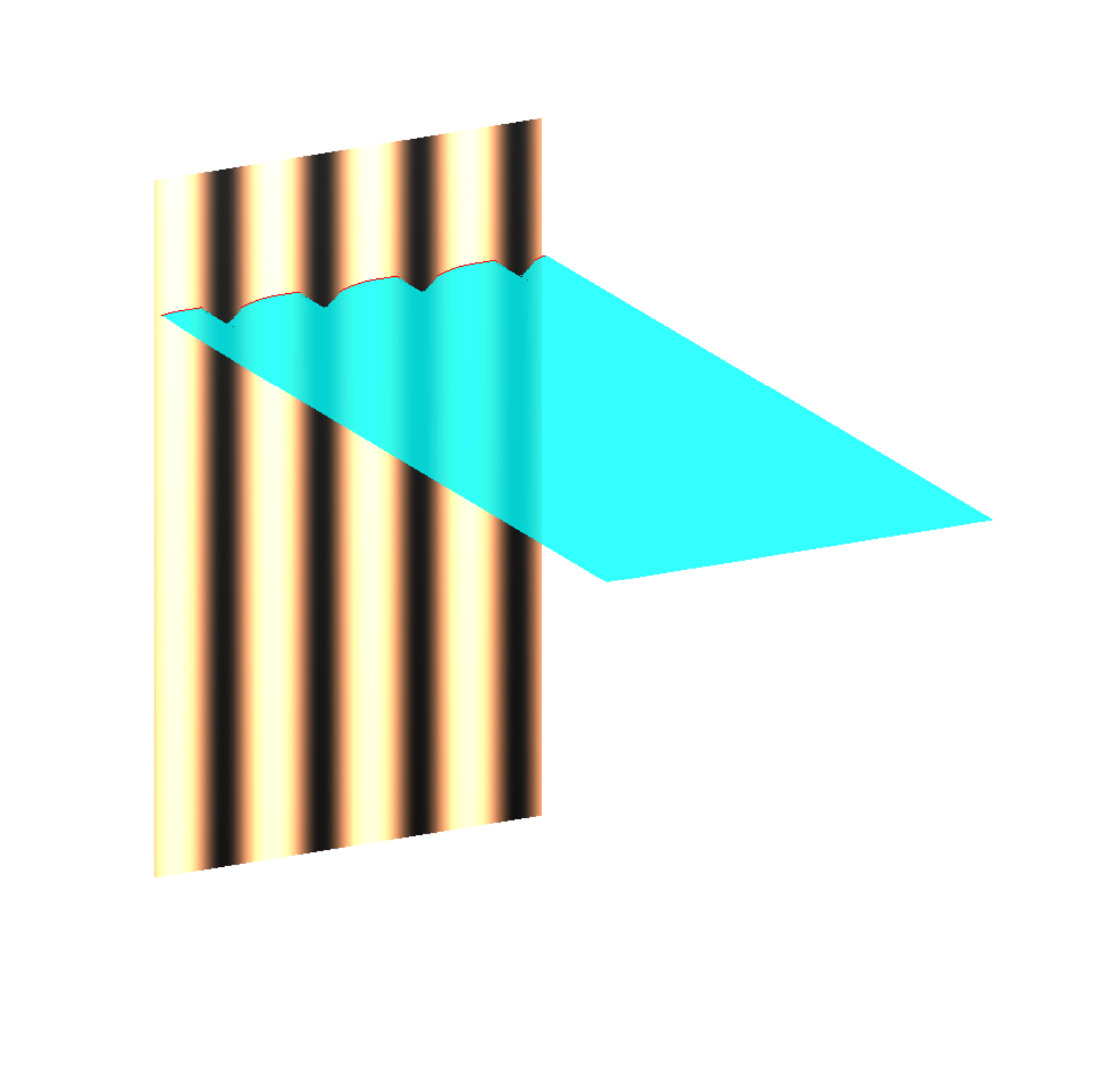}
      }
    \vspace*{-2mm}
    \caption{A flat solid surface and the fluid-fluid interfaces}    \lbl{fig:planar}
 \end{figure}

{\it The modified Cassie equation}.
If the surface is flat but inhomogeneous,  the 
solid surface is still given by $\{x=0\}$. 
This is shown in Fig.~\ref{fig:planar}(b),  where the colormap on the solid surface implies the inhomogeneity of the surface,
i.e. $\theta_s(Y,Z)$ is a periodic function. 
In this case,  we still have
$\theta_g^\eps=0$ and $\phi_\eps=0$.
The homogenized equation is reduced to
\begin{align}\label{e:modi_Cassie}
\cos\theta_a=\aint_{0}^{\eps}\cos\theta_s(y/\eps,\psi_\eps(y))\d y.
\end{align}
The equation implies that the apparent contact angle is a line average of  Young's  angle  
along the contact line. 
The equation~\eqref{e:modi_Cassie}, regarded as the modified Cassie
equation, has been derived formally in\cite{XuWang2013}. On chemically patterned surface,
i.e. the solid surface is composed by two materials, the modified Cassie equation is also 
proposed in \cite{Raj12}.
It turns out that the modified Cassie equation can describe the contact angle hysteresis on chemically inhomogeneous surface, and is consistent with some recent experiments\cite{Raj12,Priest2013}. 
 
It is easy to see the equation~\eqref{e:modi_Cassie} is different from the classical Cassie equation
\begin{equation}\label{e:Cassie}
\cos\theta_a=\aint_{0}^{\eps}\aint_{0}^{\eps}\cos\theta_s(y/\eps,z/\eps)\d y\d z,
\end{equation}
which says that the cosine of the apparent contact angle is 
an area average of the cosine of the Young's angle in chemically inhomogeneous surface.
 However, in some special situation where the energy  in the system has a unique minimizer,  the two equations can be equivalent.
To show that, we consider an example with a chemically patterned surface as in Figure~\ref{fig:planar}(c), where the Young's angle $\theta_s(y/\eps)$ is a periodic function in $y$. In this example, it is easy to see that both equations~\eqref{e:modi_Cassie} and 
\eqref{e:Cassie} reduce to the same equation
\begin{equation*}
\cos\theta_a=\aint_{0}^{\eps}\cos\theta_s(y/\eps)\d y.
\end{equation*}
More discussions on their differences and relations we refer to \cite{XuWang2013}.

{\it The modified Wenzel's equation.}
When the solid surface is chemically homogeneous but geometrically rough, 
we have  $\theta_s(y/\eps,z/\eps)\equiv\theta_Y$.  The equation \eqref{e:ModifiedWC} is reduced to
\begin{align}\label{e:modi_Wezel}
\cos\theta_a=\aint_{0}^\eps \cos\Big(\theta_Y-\theta_g^\eps (y)\Big) \sqrt{1+(\partial_y\phi_\eps)^2}\d y.
%=\aint_{0}^\eps \cos(\theta_Y-\theta_g|_{L^\eps}) \sqrt{1+(\partial_y\phi_\eps)^2}\d y.
\end{align}
This equation means that the cosine of the apparent contact angle is equal to the linear average of
the cosine of the  Young's angle subtracting a geometric angle on the contact line.
We call this equation a modified Wenzel equation.

In the following, we discuss the relation between the equation~\eqref{e:modi_Wezel} and the classical Wenzel equation
\begin{align}\label{e:class_Wenzel}
\cos\theta_a= r\cos\theta_Y,
\end{align}
where  the roughness parameter 
$$r=\aint_{0}^\eps\aint_0^\eps \sqrt{1+ (\partial_y h_\eps)^2+(\partial_z h_\eps)^2}\d y\d z$$
 is the ratio between the area of the rough surface and the effective planar surface. 
 
In some special case, the equation \eqref{e:modi_Wezel} is reduced   to the equation \eqref{e:class_Wenzel}.
For example, when the solid surface is a wave-like surface given by
$\{x=\eps h(y/\eps)\}$ for a periodic function $h(Y)$(see Fig.~\ref{fig:wavelike}(a)). In this case,
we have $\phi_\eps(y)=h_\eps(y)$ and $\theta_g^\eps=0$.
%the projection of the
%contact line $L^\eps_p$ is given by $\{ x= h_\eps(y)=\eps h(y/\eps),z=0\}$. Notice $\partial_Z h=0$, we easily see that
 %the normal $\mathbf{m}$
%is equal to $\mathbf{n}_{S}$ and
%$\theta_g=0.$. 
Thus, the equation~\eqref{e:modi_Wezel} can be reduced to
\begin{align*}
\cos\theta_a&
= \aint_0^\eps\sqrt{1+(\partial_y h_\eps)^2}\d y \cos\theta_Y
=r\cos\theta_Y,
\end{align*}
i.e. the classical Wenzel equation.
%  \begin{figure}[htb!]
%\vspace*{-2mm}
%    \centering
%      \includegraphics[height=7.5cm]{wave1.eps}
%    \vspace*{-2mm}
%    \caption{Rough surface and the fluid-fluid interfaces,  $x=\eps h(y/\eps)$}    \lbl{fig:Wavelike}
% \end{figure}

In general, the contact angle given by the modified Wenzel equation~\eqref{e:modi_Wezel}
is  different from the equation \eqref{e:class_Wenzel}.
For example, if we consider a wavelike solid
surface is given by $\{x=\eps h(z/\eps)\}$ for a periodic function $h(Z)$(see Fig.~\ref{fig:wavelike}(b)), 
it is easy to see  that a planar liquid-air interface, i.e. $u_\eps(x,y)=k(1- x)$ for some constant $k$, satisfies the equation~\eqref{e:pb3}, once the local contact angle is equal to $\theta_Y$.
In this case we have $\phi_\eps$ is a constant function and $\theta_g^{\eps}|_{L^\eps}$ is a constant
depending only on the position of the contact line. 
The equation~\eqref{e:ModifiedWC} is reduced to 
\begin{equation}\label{e:4.6}
\cos\theta_a=\cos(\theta_Y-\theta_g^{\eps}|_{L^\eps}), \hbox{or equivalently } \theta_a=\theta_Y-\theta_g^{\eps}|_{L^\eps},
\end{equation}
The situation is  more clearly  in side view. As shown in Fig.~\ref{fig:wavelike}(c), 
suppose the Young angle is equal to $90^o$, then the equation~\eqref{e:pb3} can have multiple solutions(the light blue dashed lines).
% In general, the geometric angle $\theta_g^{\eps}$ is not zero. 
 %Furthermore, Fig.~\ref{fig:wavelike}(c) also shows that
 These solutions correspond to different apparent contact angles due to different 
 $\theta_g^{\eps}$. All these angles satisfies
 the modified Wenzel equation~\eqref{e:modi_Wezel}, or equivalently \eqref{e:4.6}. 
 In comparison, 
 the equation \eqref{e:class_Wenzel} gives only
 a unique apparent angle on the surface and can not describe all the solutions.

  \begin{figure}[htb!]
\vspace*{-2mm}
    \centering
      \subfigure[wave like surface $x=\eps h(y/\eps)$]{
      \includegraphics[height=4.8cm,width=5cm]{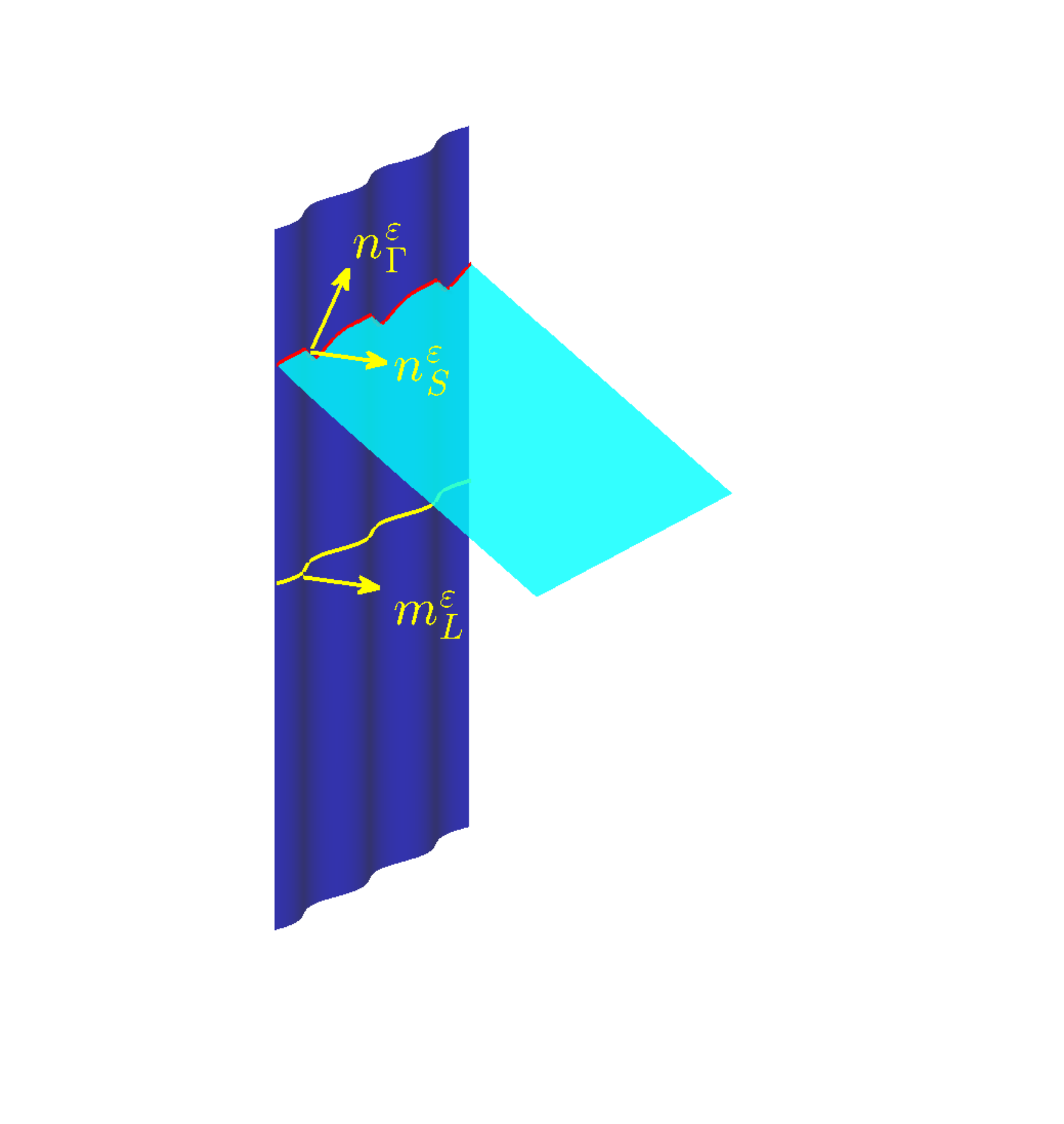}
      } 
       \subfigure[wave like surface $x=\eps h(z/\eps)$]{
      \includegraphics[height=4.8cm,width=5cm]{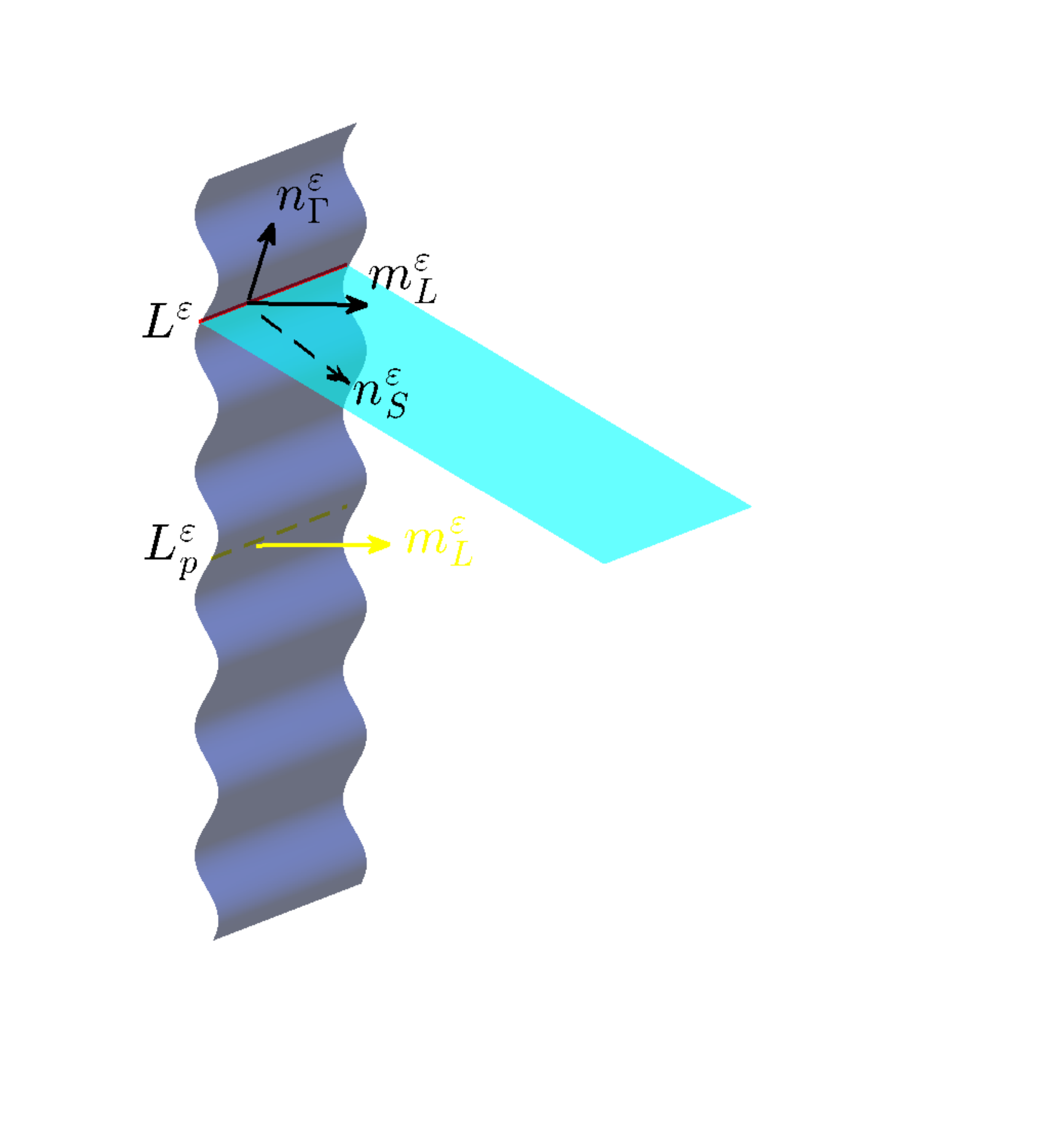}
      } 
       \subfigure[sideview of the surface $x=\eps h(z/\eps)$]{
      \includegraphics[height=4.8cm,width=5cm]{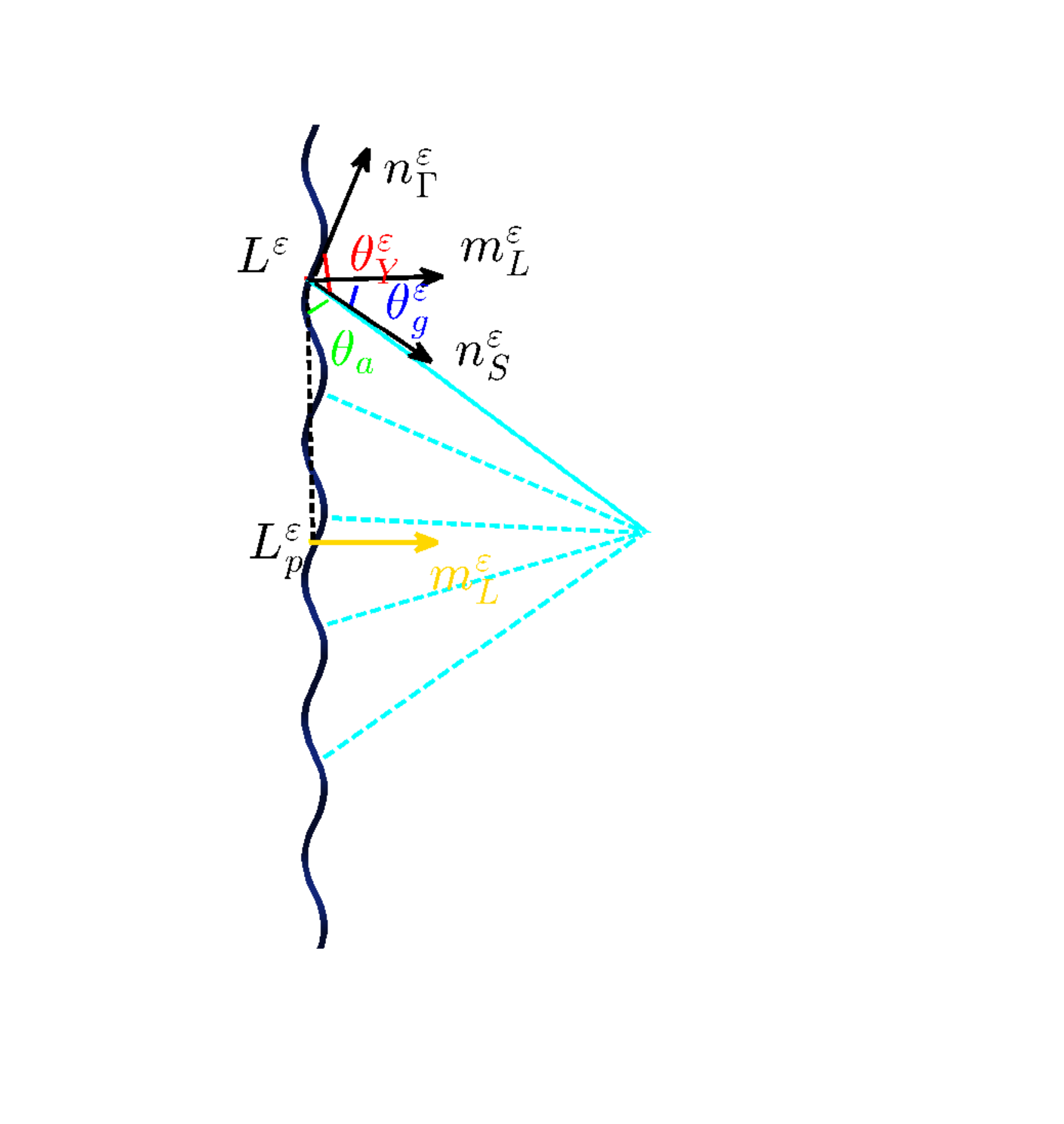}
      }
    \vspace*{-2mm}
    \caption{Rough surface and the fluid-fluid interfaces}    \lbl{fig:wavelike}
 \end{figure}
The difference between the modified Wenzel equation~\eqref{e:modi_Wezel} and the classical Wenzel equation~\eqref{e:class_Wenzel}
can be understood as follows. 
While the classical Wenzel equation corresponds to the global minimizer
of the total surface energy in the system\cite{chenWangXu2013},
the modified Wenzel equation, which is derived from the equilibrium equation~\eqref{e:pb3},
may correspond to the local minimizers in the system. Since the  minimizer is not unique in general,
the contact angle given by the modified Wenzel equation~\eqref{e:modi_Wezel} can be different from 
that given by the classical Wenzel equation \eqref{e:class_Wenzel}, as shown in the previous example.
\section{Rigorous proof}
In this section, we will prove rigorously the homogenization result by asymptotic analysis in Section~3.  
%introduce some notations. Let $u_\eps$ satisfies the solution \eqref{e:pb3} and $u_0$ is given by the equation~\eqref{e:homo1}.
For that purpose, we need the following assumption:
\begin{equation}
\label{e:assump}
\nu:=\max_{y}\Big|\cos(\theta_s^\eps+\theta_g^\eps)\sqrt{1+(\partial_y\phi_\eps)^2}\Big|<1.
\end{equation}
This assumption implies that the two-phase flow system in partial wetting regime. In addition,
to avoid technical complexity, we  assume  $h(Y,Z)\leq 0$, so that $\phi_\eps<0$.
%We suppose that $u_\eps$ is smooth. 
%By the analysis in the previous section, we could see that
%$u_\eps$ can also be written as 
%\begin{equation}\label{e:pb4}
%\left\{
% \begin{array}{ll}
%\div\Big(\frac{\nabla u_\eps}{\sqrt{1+|\nabla u_{\eps}|}}\Big)=0 ,& \hbox{ in } D^\eps \\
%\mathbf{m}_{L}\cdot\mbf{n}_{\Gamma}=\cos(\theta^\eps_Y +\theta_g^{\eps}),& \hbox{ on } L^{\eps},
% \end{array}
% \right.
%\end{equation}
%where $\theta_g^{\eps}=\theta_g(y/\eps)$ with $\theta_g$ defined in above equation.

The main result is the following theorem.
\begin{theorem}\label{theo:main}
Let $u_\eps\in C^2(D^\eps)$ be the solution of \eqref{e:pb3} and  $u_0\in C^2(0,1)$ be the homogenized function of $u_\eps$ given by \eqref{e:homo} and \eqref{e:ModifiedWC}.
 Then we have 
\begin{align}
&\max_{x\in(0,1)} \big|\bar{u}_\eps(x)-u_0(x)\big|\leq C_1\eps, \label{e:est1}
\\
&\max_{y\in (0,\eps)} \|{u}_\eps(x,y)-u_0(x)\|_{L^1(0,1)}\leq C_2\eps,\label{e:est2}
\end{align} 
 for two  constants $C_1$ and $C_2$ independent of $\eps$. Here $\bar{u}_\eps(x):=\aint_0^\eps u_\eps(x,y)\d y$.
%with $C$ depends only on the total energy of $\mathcal{E}_\eps(u_\eps)$,.
\end{theorem}

The difficulty of the proof of the theorem relies on the fact that the solution of \eqref{e:pb3} is not unique.
%In general, the definition of \eqref{e:ModifiedWC} depends also the relation of $u_\eps$.
Our proof of  Theorem~\ref{theo:main} is based on the following idea.
For any specific solution $u_\eps$ of the equation~\eqref{e:pb3},
we will establish an auxiliary variational problem,  which
has a unique minimizer. Then we utilize the variational  problem
to prove Theorem~\ref{theo:main}. 

  \begin{figure}[htb!]
\vspace*{-1mm}
    \centering
  \resizebox{!}{6.5cm}%{height=6.5cm,width=10cm}
    {\includegraphics[height=4cm]{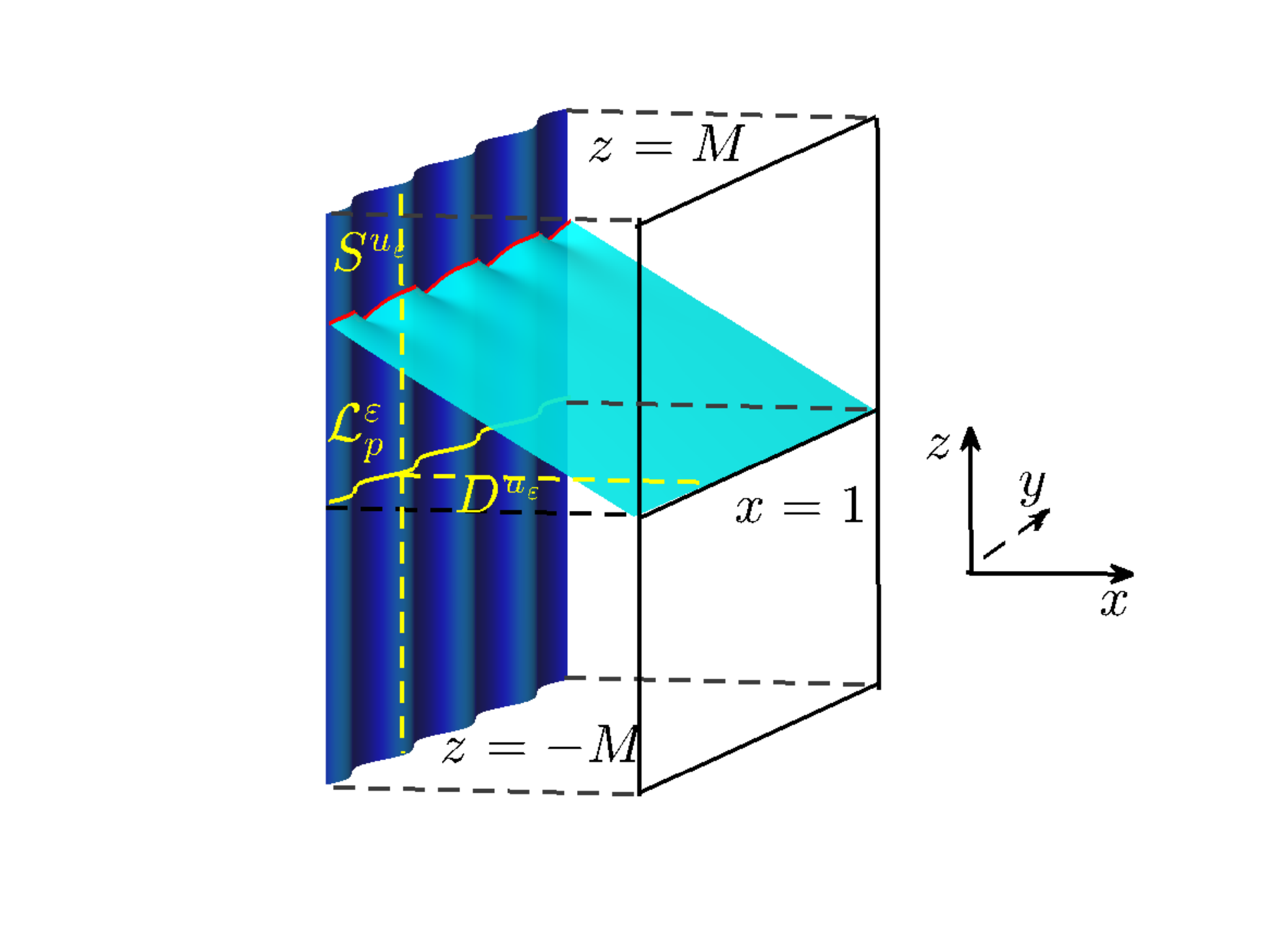}}
    \vspace*{-5mm}
    \caption{Rough surface and the fluid-fluid interfaces}    \lbl{fig:4.1}
 \end{figure}
Hereinafter, we assume that $u_\eps$ is one specific solution of the equation~\eqref{e:pb3}. 
The contact line $L^\eps$ is given by $\{x=\phi^{u_\eps}_\eps(y),z=\psi_\eps^{u_\eps}(y)\}$. Here,
to explicitly show the dependence in $u_\eps$, we use $\phi^{u_\eps}_\eps$ and $\psi^{u_\eps}_\eps$ instead of $\phi_\eps$ and $\psi_\eps$
in previous sections. Similarly, we denote $\theta^{u_\eps}_s(y):=\theta_s(y/\eps,\psi_{\eps}(y))$
and $\theta_g^{u_\eps}:=\theta_g(y/\eps,\psi_{\eps}(y))$, instead of $\theta_s^\eps$ and $\theta_g^\eps$.
We denote 
\begin{equation*}
S^{u_\eps}:=\{(x,y,z) | x=\phi^{u_\eps}_\eps(y) ,0< y<\eps, -M<z<M\}
\end{equation*}
a wave-like solid surface as shown in Fig.~\ref{fig:4.1}, and 
$$D^{u_\eps}:=\{(x,y) | \phi^{u_\eps}_\eps(y)<x<1,0<y<\eps\}$$
with the left boundary 
$$\mathcal{L}^{u_\eps}_p:=\{(x,y) | x=\phi^{u_\eps}_\eps(y),0<y<\eps\}.$$
Introduce a functional space 
\[
\mathrm{V}:=\{
v\in W^{1,2}(D^{u_\eps}) | v \hbox{ satisfies the periodic condition in $y$ and } v(1,y)=0.
\}
\]
For any $v\in\mathrm{V}$, we define an energy functional
\begin{equation}\label{e:functional}
\mathcal{E}^{u_\eps}(v)=\frac{1}{\eps}\int_{D^{u_\eps}}\!\!\!\!\sqrt{1+|\nabla v|^2}\d x\d y-
\aint_{0}^\eps {v(\phi^{u_\eps}_\eps(y),y)}\cos(\theta_s^{u_\eps}+\theta_g^{u_\eps})
\sqrt{1+(\partial_y\phi^{u_\eps}_\eps)^2}\d y.
\end{equation}
We  define a  variational  problem as follows,
\begin{equation}\label{e:aux}
\min_{v\in\mathrm{V}}\mathcal{E}^{u_\eps}(v).
\end{equation}

For the problem~\eqref{e:aux}, we first have the following lemma.
\begin{lemma}\label{lem1}
Let $u_\eps\in C^2(D^{u_\eps})$ be a solution of \eqref{e:pb3}, and $\mathcal{E}^{u_\eps}$ be defined in \eqref{e:functional}.
 Then $u_\eps$ is the  unique minimizer of the problem \eqref{e:aux}.
\end{lemma}
\begin{proof}
The proof of the lemma includes two steps. We first prove that $u_\eps$ satisfies 
the Euler-Lagrangian equation of \eqref{e:aux}, then we prove the problem \eqref{e:aux} is convex.

{\it Step 1}. The proof of the first statement is based on a standard argument.  
Suppose $u\in \mathrm{V}$ be a (local) minimizer of $\mathcal{E}^{u_\eps}$.
Then for any $w\in C^\infty(D)$ such that $w(1,y)=0$ and the periodic condition in $y$, we have
 $$\mathcal{E}^{u_\eps}(u)\leq \mathcal{E}^{u_\eps}(u+t w),\qquad\qquad \hbox{ if $|t|$ small enough}.$$
Denote $I(t)=\mathcal{E}^{u_\eps}(u+t w)$, we then have $I'(0)=0$. Notice that
\begin{align*}
I(t)&=\aint_0^{\eps}\int_{\phi^{u_\eps}_\eps}^1\!\!\sqrt{1+|\nabla u+t\nabla w|^2}\d x\d y 
\\
&\quad-\aint_{0}^{\eps}\sqrt{1+|\partial_y \phi^{u_\eps}_\eps|^2}\cos(\theta_Y^{u_\eps}+\theta_g^{u_\eps})(u(\phi^{u_\eps}_\eps,y)+t w(\phi^{u_\eps}_\eps,y))\d y.\nonumber\\
\end{align*} 
 Direct computations give
\begin{align}
I'(0)=&\aint_0^{\eps}\int_{\phi^{u_\eps}}^1\frac{\nabla u\cdot\nabla w}{\sqrt{1+|\nabla u|^2}}\d x\d y
&-\aint_0^{\eps}\sqrt{1+|\partial_y\phi^{u_\eps}|^2}\cos(\theta_Y^{u_\eps}+\theta_g^{u_\eps}) w(\phi^{u_\eps},y)\d y
\end{align}
For the first term in the right hand side, integration by part leads to
\begin{align}
&\aint_0^{\eps}\int_{\phi^{u_\eps}(y)}^1\frac{\nabla u\cdot\nabla w}{\sqrt{1+|\nabla u|^2}}\d x\d y\nonumber\\
=&
-\frac{1}{\eps}\int_{\mathcal{L}_{p}^{u_\eps}}\frac{\nabla u}{\sqrt{1+|\nabla u|^2}}\cdot \frac{(1,-\partial_Y\phi_\eps^{u_\eps})}
{\sqrt{1+(\partial_Y\phi^{u_\eps}_\eps)^2}} w\d s-\aint_0^{\eps}\int_{\phi_\eps^{u_\eps}}^1
\div\left(\frac{\nabla u}{\sqrt{1+|\nabla u|^2}}\right)w \d x\d y\nonumber\\
=&\aint_0^\eps\sqrt{1+(\partial_y \phi^{u_\eps} )^2}\mbf n_{\Gamma}^{\eps}\cdot\mbf m^{\eps}_{L} w\d y-\aint_0^{\eps}\int_{\phi^{u_\eps}}^1
\div\left(\frac{\nabla u}{\sqrt{1+|\nabla u|^2}}\right)w \d x\d y
\end{align}
where $\mbf n_\Gamma^{\eps}$  and  $\mbf m_{L}^{\eps}$ are defined in \eqref{e:normal} and \eqref{e:normalCL}. Therefore, $I'(0)=0$ 
implies 
\begin{align*}
-\aint_0^{\eps}\int_{\phi_\eps^{u_\eps}}^1
&\div\left(\frac{\nabla u}{\sqrt{1+|\nabla u|^2}}\right)w \d x\d y\\
&+
\aint_0^\eps\sqrt{1+(\partial_y \phi_\eps^{ u_\eps} )^2}(\mbf n^\eps_{\Gamma}\cdot\mbf m^\eps_{L}-\cos(\theta_Y^{u_\eps}+\theta_g^{u_\eps}) ) w\d y=0,
\end{align*}
for all $w$. 
Therefore, the Euler-Lagrangian equation of the problem~\eqref{e:aux} is given by
\begin{align}
&\div\left(\frac{\nabla u}{\sqrt{1+|\nabla u|^2}}\right)=0, \qquad \hbox{ in } D^{u_\eps} \label{e:curvature}\\
&\mbf n^\eps_{\Gamma}\cdot\mbf m^\eps_{L}-\cos(\theta_Y^{u_\eps}+\theta_g^{u_\eps})=0, \qquad \hbox{ on } \mathcal{L}^{u_\eps}_p.
\label{e:tempbnd}
\end{align}
Notice the definition of $\theta_g^{u_\eps}$ in \eqref{e:geoAng}, the equation \eqref{e:tempbnd} is equivalent to 
\begin{equation*}
\mbf n_{\Gamma}^{\eps}\cdot\mathbf{n}_{S}^{\eps}=\cos \theta_Y^{u_\eps},  \qquad \hbox{ on } \mathcal{L}^{u_\eps}_p.
\end{equation*}
By equation~\eqref{e:pb3}, it is easy to see that $u_\eps$ satisfies \eqref{e:curvature} and the above boundary condition.
%This implies that $u$ satisfies \eqref{e:pb4}.

(2). To finish the proof of the lemma, we need only prove the convexity of
the functional $\mathcal{E}^{u_\eps}$. Notice that $\sqrt{1+t^2}$ is strictly convex. Then for any $u_1,u_2\in \mathrm{V}$, $u_1\neq u_2$, and $0<t<1$, let $u^t=t u_1+(1-t)u_2$. Then we have
\begin{align}
\mathcal{E}^{u_\eps}(u^t)&=\aint_0^\eps\int_{\phi^{u_\eps}}^1\sqrt{1+|t\nabla u_1+(1-t)\nabla u_2|^2}\d x\d y\nonumber\\
&\quad
-\aint_0^\eps\sqrt{1+(\partial_y \phi^{u_\eps} )^2}\cos(\theta_Y^{u_\eps}+\theta_g^{u_\eps})(tu_1+(1-t)u_2)\d y \nonumber\\
&<t\mathcal{E}^{u_\eps}(u_1)+(1-t)\mathcal{E}^{u_\eps}(u_2).\nonumber
\end{align}
\end{proof}

In the following, we  estimate the difference between $u_\eps$ and $u_0$ by utilizing the energy minimizing problem~\eqref{e:aux}. Similar technique has been used to verify the Wenzel equation in \cite{chenWangXu2013}.

We firstly estimate the difference between $v\in \mathrm{V}$ and its average:
\begin{align}\label{e:ave}
\bar{v}:=\aint_0^\eps v(x,y)\d y.
\end{align}
This is given by the following lemma.
\begin{lemma}\label{lem2}
For any $v \in\mathrm{V}$, $\bar{v}$ is defined as \eqref{e:ave}, we have
\begin{align*}
\max_{y\in(0,\eps)} \|v-\bar{v}\|_{L^1(0,1)}\leq\frac{\mathcal{E}^{u_\eps}(v)}{1-\nu}\eps.
\end{align*} 
where $\nu$ is defined in~\eqref{e:assump}.
\end{lemma}
\begin{proof}
For each $y\in(0,\eps)$, we have
\begin{align}
\int_0^1|\bar{v}(x)-v(x,y)|\d x 
&= \int_0^1 |\aint_{0}^\eps v(x,\hat{y})-v(x,y)\d \hat{y}|\d x\nonumber\\
&\leq \int_0^1\aint_0^\eps\int_0^{\eps}|\partial_y v(x,\tilde{y})|\d \tilde{y}\d \hat{y}\d x\nonumber\\
&=\int_0^1\int_0^{\eps}|\partial_y v(x,\tilde{y})|\d \tilde{y}\d x \leq\eps\aint_0^\eps\int_0^1\sqrt{1+|\nabla v|}\d x\d \tilde{y}.\label{e:aveTMP}
\end{align}
Notice that
\begin{align*}
&\Big|\aint_0^\eps \sqrt{1+(\partial_y\phi^{u_\eps})^2} v \cos(\theta_Y^{u_\eps}+\theta_g^{u_\eps}) \d y \Big| \\
\leq & \nu\aint_0^\eps |v(\phi^{u_\eps}(y),y)|\d y=\nu\aint_0^{\eps}|\int_{\phi^{u_\eps}}^1\partial_x v(x,y)\d x|\d y \\
\leq &\nu\aint_0^{\eps}\int_{\phi^{u_\eps}}^1 |\partial_x v(x,y)|\d x\d y
<\nu\int_0^1\int_{\phi^{u_\eps}}^1\sqrt{1+|\nabla v|^2}\d x \d y.
\end{align*}
This leads to
\begin{align}\label{e:engTMP}
\mathcal{E}^{u_\eps}(v)>(1-\nu)\aint_0^\eps\int_{\phi^{u_\eps}(y)}^1\sqrt{1+|\nabla v|^2}\d x\d y.
\end{align}
Combining the estimate \eqref{e:aveTMP} and \eqref{e:engTMP}, we  finish the proof of the lemma.
\end{proof}

The next lemma is a technical result, 
which estimates the difference between  $\bar{v}(0)$ and $v$ on the left boundary.
\begin{lemma}\label{lem:nearBnd}
For any $v \in\mathrm{V}$, we have
\begin{align}\label{e:nearBnd}
\aint_0^\eps|v(\phi^{u_\eps}(y),y)-\bar{v}(0)|\d y\leq \frac{2}{1-\nu}\Big( (1+\|h\|_{\infty})\eps+\mathcal{E}^{u_\eps}(v_\eps) - C_\eps \Big),
\end{align}
with $C_\eps=\inf_{v}\mathcal{E}^{u_\eps}(v)=\mathcal{E}^{u_\eps}(u_\eps)$,
and $\|h\|_{\infty}=\max_{Y,Z}h(Y,Z)$.
\end{lemma}
\begin{proof}
Firstly, we have
\begin{align*}
\aint_0^\eps|v(\phi^{u_\eps}(y),y)-\bar{v}(0)|\d y\leq & 
\aint_0^\eps\Big|v(\phi^{u_\eps}(y),y)-\aint_0^\eps v(x,y)\d x \Big|\d y\\
&+\aint_0^\eps\Big |\aint_0^\eps (v(x,y) -\bar{v}(x))\d x\Big|\d y
+\Big|\aint_0^\eps \bar{v}(x)\d x-\bar{v}(0)\Big|\\
=:&I_1+I_2+I_3.
\end{align*}
We can estimate $I_i$ term by term. Notice that $\phi^{u_\eps}<0$, we have
\begin{align*}
I_1=\aint_0^\eps\Big|\aint_0^\eps v(\phi^{u_\eps}(y),y)-v(x,y)\d x \Big|\d y
&\leq \aint_0^\eps \aint_0^\eps \int_{\phi^{u_\eps}}^\eps|\partial_x v(\tilde{x},y)|\d \tilde{x}\d x \d y\\
&=\aint_0^\eps \int_{\phi^{u_\eps}}^\eps|\partial_x v(\tilde{x},y)|\d \tilde{x} \d y,
\end{align*}
\begin{align*}
I_2=\aint_0^\eps\Big |\aint_0^\eps \aint_0^\eps v(x,y) - {v}(x,\tilde{y})\d\tilde{y}\d x\Big|\d y\leq 
\aint_0^\eps \aint_0^\eps |\partial_y v (x,\hat{y})|\d\hat{y}\d x,
\end{align*}
and 
\begin{align*}
I_3=\Big|\aint_0^\eps \aint_0^\eps {v}(x,y)-{v}(0,y)\d y \d x\Big|
&\leq 
\aint_0^\eps \aint_0^\eps \int_0^{\eps} |\partial_x {v} (\hat x,y)|\d \hat{x}\d y \d x\\
&=\aint_0^\eps \int_0^{\eps} |\partial_x {v} (\hat x,y)|\d \hat{x}\d y.
\end{align*}
Then we have 
\begin{align}\label{e:tempI4}
\aint_0^\eps|v(\phi^{u_\eps}(y),y)-\bar{v}(0)|\d y\leq 2\aint_{0}^{\eps}\int_{\phi^{u_\eps}(y)}^\eps \sqrt{1+|\nabla v|^2}\d x \d y=:I_4.
\end{align}
We need only estimate the  term $I_4$.

For  $v\in \mathrm{V}$, we define a  function by translation
$$
\mathrm{T}_\delta v :=\left\{
\begin{array}{ll}
0 & \hbox{if } x\in[1-\delta,1],\\
v(x+\delta) & \hbox{if } x\in[\phi^{u_\eps}(y), 1-\delta].
\end{array}
\right.
$$
We can easy see that $\mathrm{T}_\delta v_\eps\in X$. Direct computations give
\begin{align*}
\mathcal{E}^{u_\eps}(v)-\mathcal{E}^{u_\eps}( \mathrm{T}_\delta v)=&
\aint_0^\eps\Big(\int_{\phi^{u_\eps}(y)}^{\phi^{u_\eps}+\delta}\sqrt{1+|\nabla v|^2}\d x -
\int_{1-\delta}^1\sqrt{1+|\nabla T_\delta v|^2}\d x \Big)  \d y \\
 &-
\aint_0^{\eps}\sqrt{1+(\partial_y\phi^{u_\eps})^2}(v(\phi^{u_\eps}(y),y)-v(\phi^{u_\eps}(y)+\delta,y )\cos(\theta_s^{u_\eps}-\theta_g^{u_\eps})\d y\\
\geq & \aint_0^\eps \int_{\phi^{u_\eps}(y)}^{\phi^{u_\eps}+\delta}\sqrt{1+|\nabla v|^2}\d x \d y
-\delta
-\nu \aint_0^{\eps}\int_{\phi^{u_\eps}(y)}^{\phi^{u_\eps}+\delta}|\partial_x v(x,y)|\d x\d y\\
\geq &(1-\nu ) \aint_0^\eps \int_{\phi^{u_\eps}(y)}^{\phi^{u_\eps}+\delta}\sqrt{1+|\nabla v|^2}\d x-\delta
\end{align*}
Setting $\delta=(1+\|h\|_{\infty})\eps$, and using $\mathcal{E}_\eps( \mathrm{T}_\delta v_\eps)\geq C_\eps$,
the above equation and \eqref{e:tempI4} implies \eqref{e:nearBnd}.
\end{proof}

%%%%%%%%%
We now introduce an energy minimizing problem for the homogenized problem. Define a functional space,
\begin{align*}
\mathrm{V}_1:=\{v\in W^{1,2}(0,1) : v(1)=0\}.
\end{align*}
For any $v(x)\in \mathrm{V}_1$, we define, 
\begin{equation}\label{e:homoEng}
E(v)=\int_0^1\sqrt{1+v_x^2}\d x- v(0)\cos\theta_a,
\end{equation}
where $\theta_a=\arccos\Big(\aint_0^\eps\sqrt{1+\partial_y\phi^{u_\eps}}\cos(\theta_Y^{u_\eps}+\theta_g^{u_\eps})\d y\Big).$
The following lemma bounds the energy $E(\bar{u}_\eps)$ by $\mathcal{E}^{u_\eps}(u_\eps)$.
\begin{lemma}\label{lem3}
Let $\bar{u}_\eps$ be the average of $u_\eps$. We have
\begin{align*}
E(\bar{u}_\eps)\leq \mathcal{E}^{u_\eps}(u_\eps)+\frac{2\nu}{1-\nu}(1+\|h\|_{\infty})\eps.
\end{align*}
\end{lemma}
\begin{proof}
For each $x\in [0,1]$, it is easy to see that
\[
|\partial_x \bar{u}_\eps(x)|=\Big| \aint_0^\eps\partial_x u_\eps(x,y)\d y\Big|\leq \aint_0^\eps|\partial_x u_\eps(x,y)|\d y=:s_0.
\]
By convexity of $f(s)=\sqrt{1+s^2}$, we have
\begin{align*}
\aint_0^\eps\sqrt{1+|\partial_x u_\eps|^2}\d y=\aint_0^\eps f(|\partial_x u_\eps|)\d y
&\geq \aint_0^\eps f(s_0)+f'(s_0)(|\partial_x u_\eps|-s_0)\d y\\
&=f(s_0)=\sqrt{1+(\partial_x \bar{u}_\eps)^2}.
\end{align*}
This leads to
\begin{align}\label{e:temp1}
\aint_0^\eps\int_{\phi_\eps}^1\sqrt{1+|\nabla u_\eps|^2}\d x\d y\geq 
\aint_0^\eps\int_{0}^1\sqrt{1+|\partial_x u_\eps|^2}\d x\d y\geq \int_0^1\sqrt{1+|\partial_x \bar{u}_\eps|^2}\d x.
\end{align}

The surface term in $E(\bar{u}_\eps)$ 
is given by
\begin{align*}
I_1=-\bar{u}_\eps(0)\cos\theta_a=-\bar{u}_\eps(0)\aint_0^\eps \sqrt{1+(\partial_y\phi^{u_\eps})^2}\cos(\theta_Y^{u_\eps}+\theta_g^{u_\eps})\d y,
\end{align*}
and the surface energy term in $\mathcal{E}^{u_\eps}(u_\eps)$ is given by
\begin{align*}
I_2=-\aint_0^\eps \sqrt{1+(\partial_y\phi^{u_\eps})^2}\cos(\theta_Y^{u_\eps}+\theta_g^{u_\eps}) 
u_\eps(\phi^{u_\eps}(y),y)\d y.
\end{align*}
We easily see that
\begin{align*}
|I_1-I_2|&=\Big|\aint_0^\eps \sqrt{1+(\partial_y\phi^{u_\eps})^2}\cos(\theta_Y^{u_\eps}+\theta_g^{u_\eps}) (u_\eps(\phi^{u_\eps}(y),y)-\bar{u}_\eps(0))\d y \Big|\\
&\leq \nu\aint_0^\eps|u_\eps(\phi^{u_\eps}(y),y)-\bar{u}_\eps(0)|\d y.
\end{align*}
By Lemma~\ref{lem:nearBnd}, noticing that $C_\eps=\mathcal{E}^{u_\eps}(u_\eps)$,  we have
\[
|I_1-I_2|\leq \frac{2\nu}{1-\nu}(1+\|h\|_{\infty})\eps.
\]
Combining the equation with \eqref{e:temp1}, we finish the proof of the lemma.
\end{proof}

The next lemma characterize the minimizer of the functional $E(v)$.
\begin{lemma}\label{lem4}
Let   $u_0$ be the homogenized function of $u_\eps$ given by \eqref{e:homo} and \eqref{e:ModifiedWC}. Then $u_0$ is the unique minimizer of $E(v)$
in $\mathrm{V}_1$, and satisfies
\begin{align*}
C_\eps-\eps\|h\|_{\infty}\leq E(u_0)\leq 1,
\end{align*}
where $C_\eps=\inf_{v\in\mathrm{V}}\mathcal{E}^{u_\eps}(v)=\mathcal{E}^{u_\eps}(u_\eps)$.
\end{lemma}
\begin{proof}
By convexity of $E(v)$, we know that $E(v)$ has a unique minimizer, satisfying the Euler-Lagrangian equation
\begin{align*}
\left\{
\begin{array}{ll}
\partial_x\Big(\frac{\partial_x v}{1+|\partial_x  v|} \Big)=0, & \hbox{ in } (0,1), \\
-\frac{\partial_x v}{\sqrt{1+|\partial_x v|^2}}=\cos\theta_a, & \hbox{ at  } x=0,\\
v(1)=0.
\end{array}
\right.
\end{align*}
It is easy to see that $u_0=k(1-x)$ with $k$ such that $\frac{k}{\sqrt{1+k^2}}=\cos\theta_a$ is a solution of the above equation. So it is the unique minimizer of $E(v)$ in 
$\mathrm{V}_1$. Furthermore, 
\[
E(u_0)=\sqrt{1+k^2}-k\cos\theta_a=\frac{1}{\sqrt{1+k^2}}\leq 1.
\]

In addition, for any $v\in \mathrm{V}_1$, define 
$$\tilde{v}(x,y):=\left\{ 
\begin{array}{ll}
v(x) & \hbox{if } (x,y)\in D^{u_\eps}, 0\leq x\leq 1\\
v(0) & \hbox{if } (x,y)\in D^{u_\eps},  x<0.
\end{array}
\right.
$$
Then $\tilde{v}\in\mathrm{V}$, and we have
\begin{align*}
\mathcal{E}^{u_\eps}(\tilde{v})&=\aint_0^\eps\int_{\phi^{u_\eps}}^1\sqrt{1+(\partial_x \tilde{v})^2}\d x\d y
-\aint_0^\eps \sqrt{1+(\partial_y\phi^{u_\eps})^2} v(0)\cos(\theta_Y^{u_\eps}+\theta_g^{u_\eps})\d y\\
&\leq \int_0^1\sqrt{1+(\partial_x v)^2}\d x+\|\phi^{u_\eps}\|_{\infty} -v(0)\cos\theta_a\\
&\leq  E(v)+\eps\|h \|_{\infty}.
\end{align*}
We then have that 
$E(u_0)\geq \inf_{v}\mathcal{E}^{u_\eps}(v)-\eps\|h\|_{\infty}$.
\end{proof}

Using these lemmas, we are ready to prove the main theorem.

{\it Proof of Theorem~\ref{theo:main}.}
By Lemma~\ref{lem1} and Lemma~\ref{lem4}, we know that $u_\eps$ is the unique minimizer of $\mathcal{E}_\eps$ in $\mathrm{V}$,
and $u_0$ is the unique minimizer of $E(v)$ in $\mathrm{V}_1$.

Notice that 
\begin{align}
\|u_\eps(\cdot,y)-u_0\|_{L^1(0,1)}&\leq \|u_\eps(\cdot,y)-\bar{u}_\eps\|_{L^1(0,1)}+\|\bar{u}_\eps -u_0\|_{L^1(0,1)}
\nonumber\\
&\leq \frac{\mathcal{E}_\eps(u_\eps)}{1-\nu}\eps+\|\bar{u}_\eps -u_0\|_{L^1(0,1)}\nonumber\\
&\leq \frac{1+\eps\|h\|_{\infty}}{1-\nu}\eps+\max_{x\in(0,1)}\big|\bar{u}_\eps -u_0\big|.\label{e:theo_temp1}
\end{align}
where we use Lemma~\ref{lem2} and Lemma~\ref{lem4}.
In the following, 
we need only prove \eqref{e:est1}, i.e. to estimate $\max_{x\in(0,1)}\big|\bar{u}_\eps -u_0\big|\leq C_2\eps$, since the equation
 implies \eqref{e:est2} by setting $C_1=C_2+\frac{1+\eps\|h\|_{\infty}}{1-\nu}$.
 
By Lemma~\ref{lem3} and Lemma~\ref{lem4}, we have 
\begin{align*}
E(u_0)\leq E(\bar{u}_\eps)&\leq \mathcal{E}^{u_\eps} (u_\eps)+\frac{2\nu}{1-\nu}(1+\|h\|_{\infty})\eps\\
%&\leq (1+\eps\|h\|_{\infty})E(u_0) +\frac{2\nu}{1-\nu}(1+\|h\|_{\infty})\eps\\
&\leq E(u_0) +\frac{2\nu}{1-\nu}(1+\|h\|_{\infty})\eps+\eps\|h\|_{\infty}\leq E(u_0) 
+\frac{1+\nu}{1-\nu}(1+\|h\|_{\infty})\eps.
\end{align*}
So
\begin{align*}
0\leq E(\bar{u}_\eps)-E(u_0)\leq \frac{1+\nu}{1-\nu}(1+\|h\|_{\infty})\eps.
\end{align*}
Notice $\bar{u}_\eps$ and $u_0$ are both continuous functions in $(0,1)$. Let $x^*\in(0,1)$ 
be the point where $|\bar{u}_\eps(x)-u_0(x)|$ attains its maximum value. Define a piecewise 
linear function $w(x)$ satisfies 
\begin{align*}
w(0)=\bar{u}_\eps(0), \qquad w(1)=\bar{u}_\eps(1), \qquad w(x^*)=\bar{u}_\eps(x^*).  
\end{align*}
Then we easily have
\begin{align*}
E(u_0)\leq E(w)\leq E(\bar{u}_\eps).
\end{align*}
This gives
\begin{align*}
0\leq E({w})-E(u_0)\leq \frac{1+\nu}{1-\nu}(1+\|h\|_{\infty})\eps,
\end{align*}
and leads to $|w(x^*)-u_0(x^*)|\leq C_2\eps$, with 
$C_2=\frac{1+\nu}{1-\nu}(1+\|h\|_\infty)\eps$. This implies \eqref{e:est1}.
$\square$

\section{General discussions}
In this section, we  discuss  how to use  the modified Wenzel and Cassie equations (or the 
formula~\eqref{e:ModifiedWC}) in practice. 
In equation~\eqref{e:ModifiedWC}, the position of 
the contact line is not known a priori, since
it is usually difficult to predict which local minimizer a real system will finally arrive at.
However, there  are  some possible ways to use the formula. Firstly,
it is possible that the position of the contact line can be experimentally
determined\cite{erbil2014debate}. In this case, one can use the equation
directly. Secondly, in many cases, people may be only interested in 
 contact angle hysteresis, i.e.
the largest and the smallest apparent contact angles
in the system. 
Then, one may use our formula to obtain the two angles by
checking for some possible contact lines\cite{XuWang2013}.
In the following, we will show this by a few simple examples.

We first consider a two-dimensional drop on a solid surface.  
In this case, the contact line becomes a point and the equation \eqref{e:ModifiedWC} reduces to
\begin{equation}\label{e:2D}
\theta_{a}=\theta_s(x_{ct})-\theta_g(x_{ct}), 
\end{equation}
where $x_{ct}$ is the contact point.
We will use this 
equation to quantify the contact angle hysteresis in the system.
For simplicity, we suppose the solid surface is either chemically patterned (see Fig.~\ref{fig:chem_CAH}) 
or geometrically rough (see Fig.~\ref{fig:rough_CAH}).
In the former case, the solid surface is composed by two materials,
with different Young's angles $\theta_{Y1}$ and $\theta_{Y2}$ ($\theta_{Y1}>\theta_{Y2}$). Since the solid
surface is flat, the geometric angle $\theta_g$ in equation \eqref{e:2D} is $0$.
The equation implies that the apparent contact angle 
$\theta_a$ is either $\theta_{Y1}$ or  $\theta_{Y2}$ depending 
on the location the contact point. By the definition of contact angle hysteresis,
%(the difference between the advancing contact angle and the receding one)
we easily see that the largest apparent contact angle(or the advancing angle) is $\theta_{Y1}$
and the smallest contact angle(or the receding angle) is $\theta_{Y2}$, as shown in Fig.~\ref{fig:chem_CAH}.
In geometrically rough surface case,  Young's angle $\theta_Y$ is a constant.
The equation \eqref{e:2D} implies that the apparent contact angle is $\theta_Y$ minus a geometric 
angle, which by definition is equal to the angle between the tangential line of the solid surface and the horizontal
effective surface. The geometric angle may be positive or negative depending on the relative position of the tangential line with 
respect to the horizontal surface. By careful computations, we could 
see that the largest apparent contact angle is $\theta_Y+\theta_{g1}$ and the smallest one is $\theta_Y-\theta_{g2}$,
where $\theta_{g1}$ and $\theta_{g2}$ are positive numbers, as shown in Fig.~\ref{fig:rough_CAH}.
{
  \begin{figure}[htb!]
\vspace*{-5mm}
    \centering
    {\includegraphics[height=5.1cm]{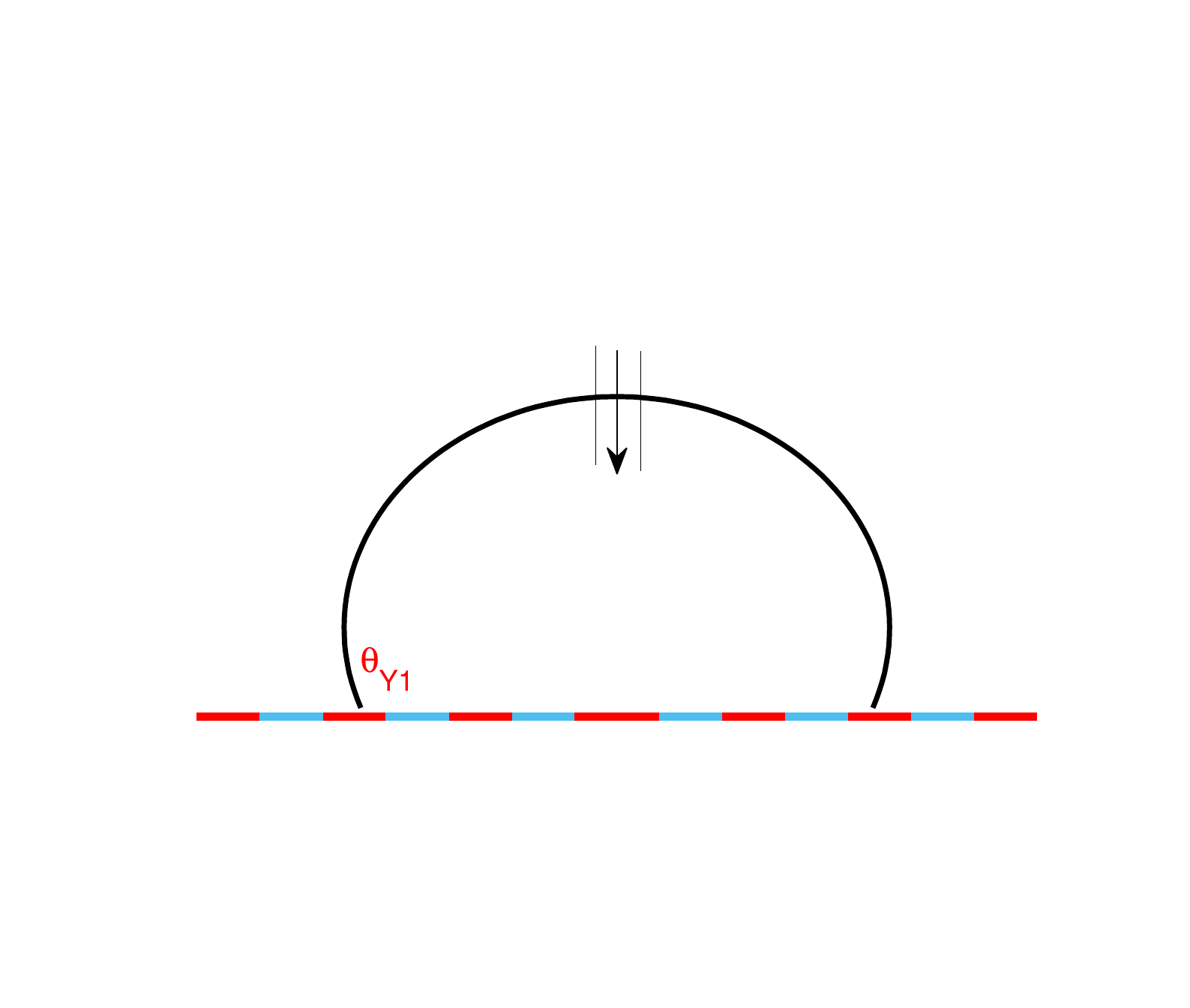}}
    {\includegraphics[height=5.3cm]{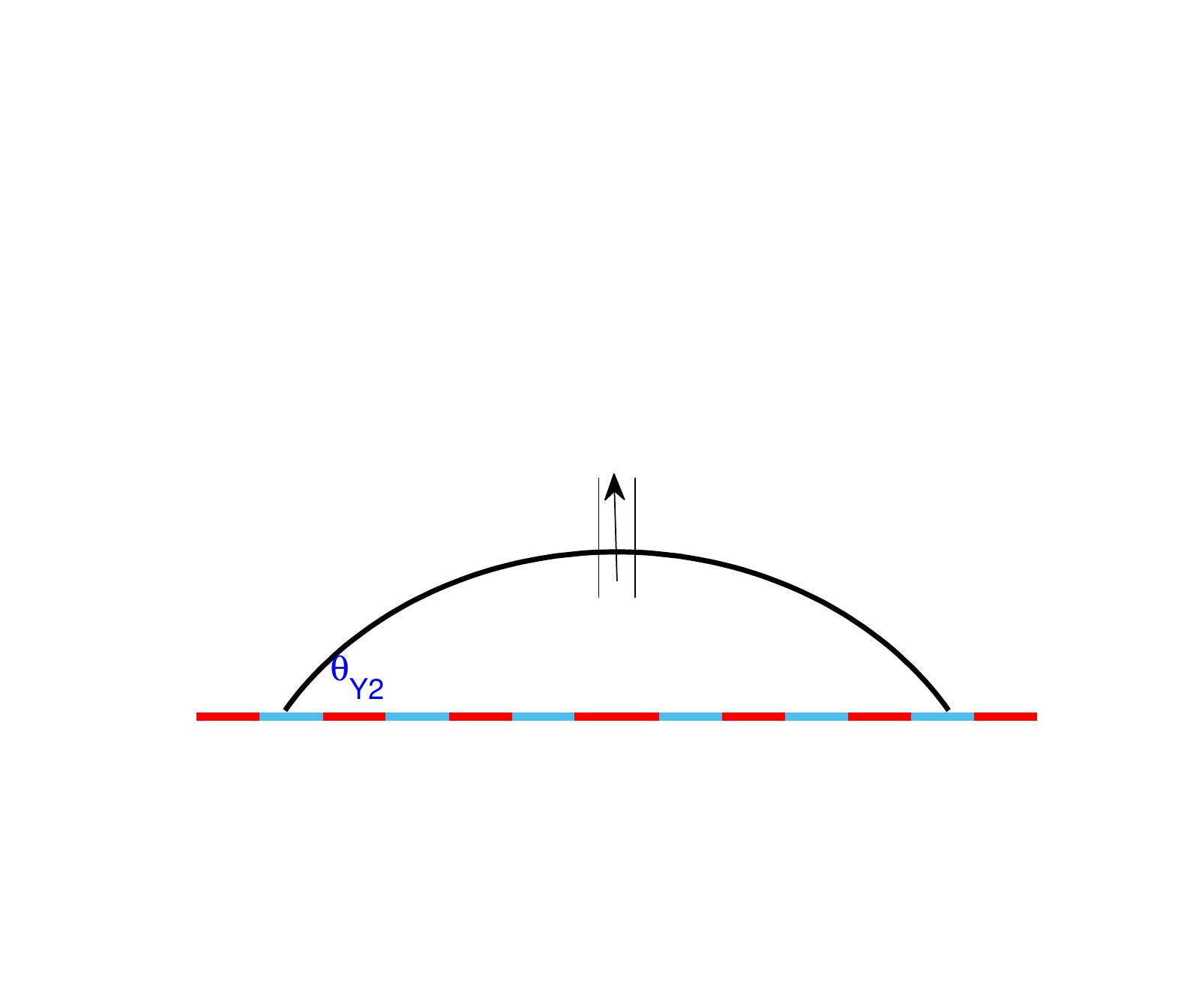}}
    \vspace*{-10mm}
    \caption{The advancing and receding contact angles on a chemically patterned surface.}    \lbl{fig:chem_CAH}
 \end{figure}
  \begin{figure}[htb!]
\vspace*{-5mm}
    \centering
    {\includegraphics[height=5.cm]{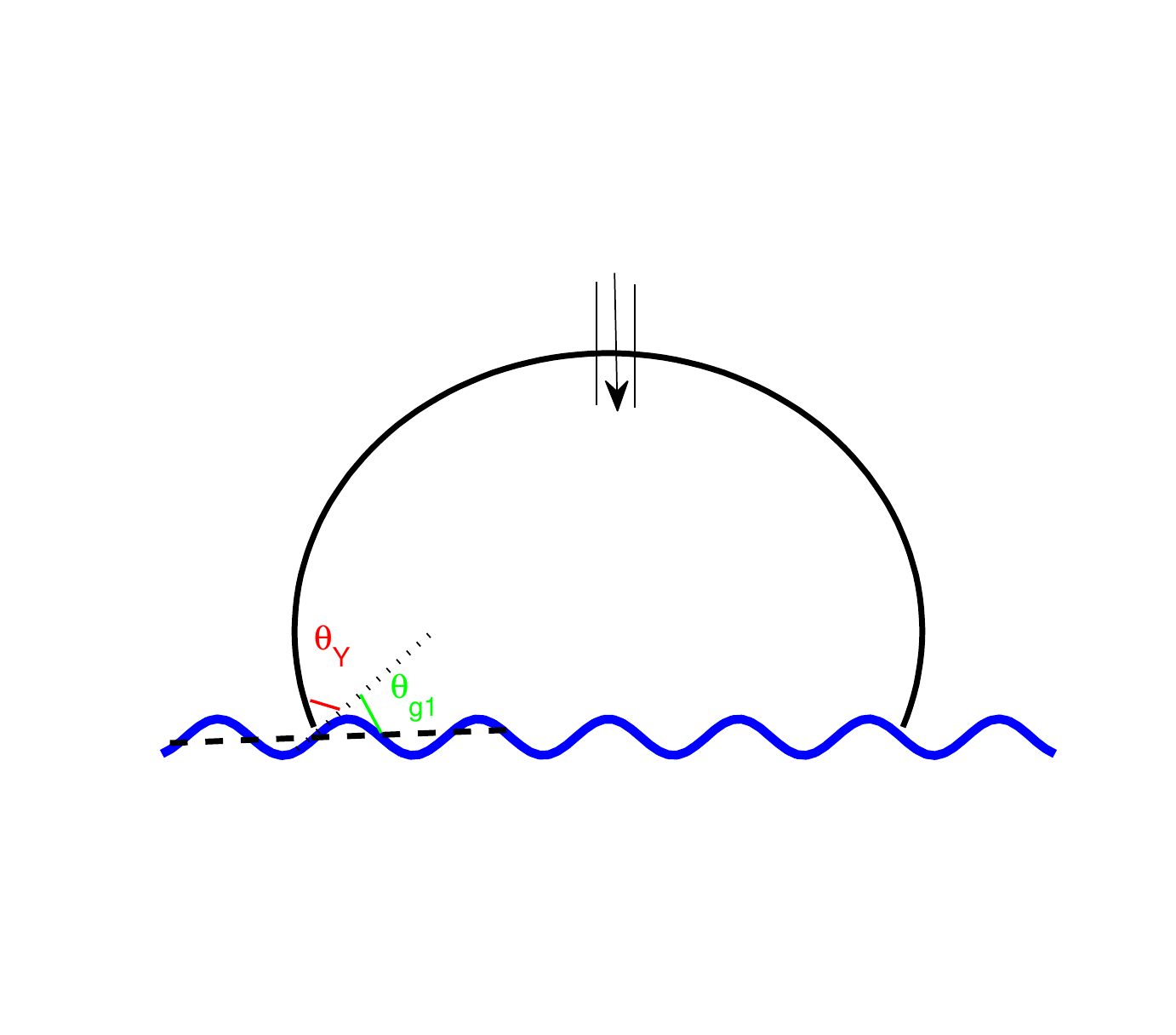}}
    {\includegraphics[height=5.3cm]{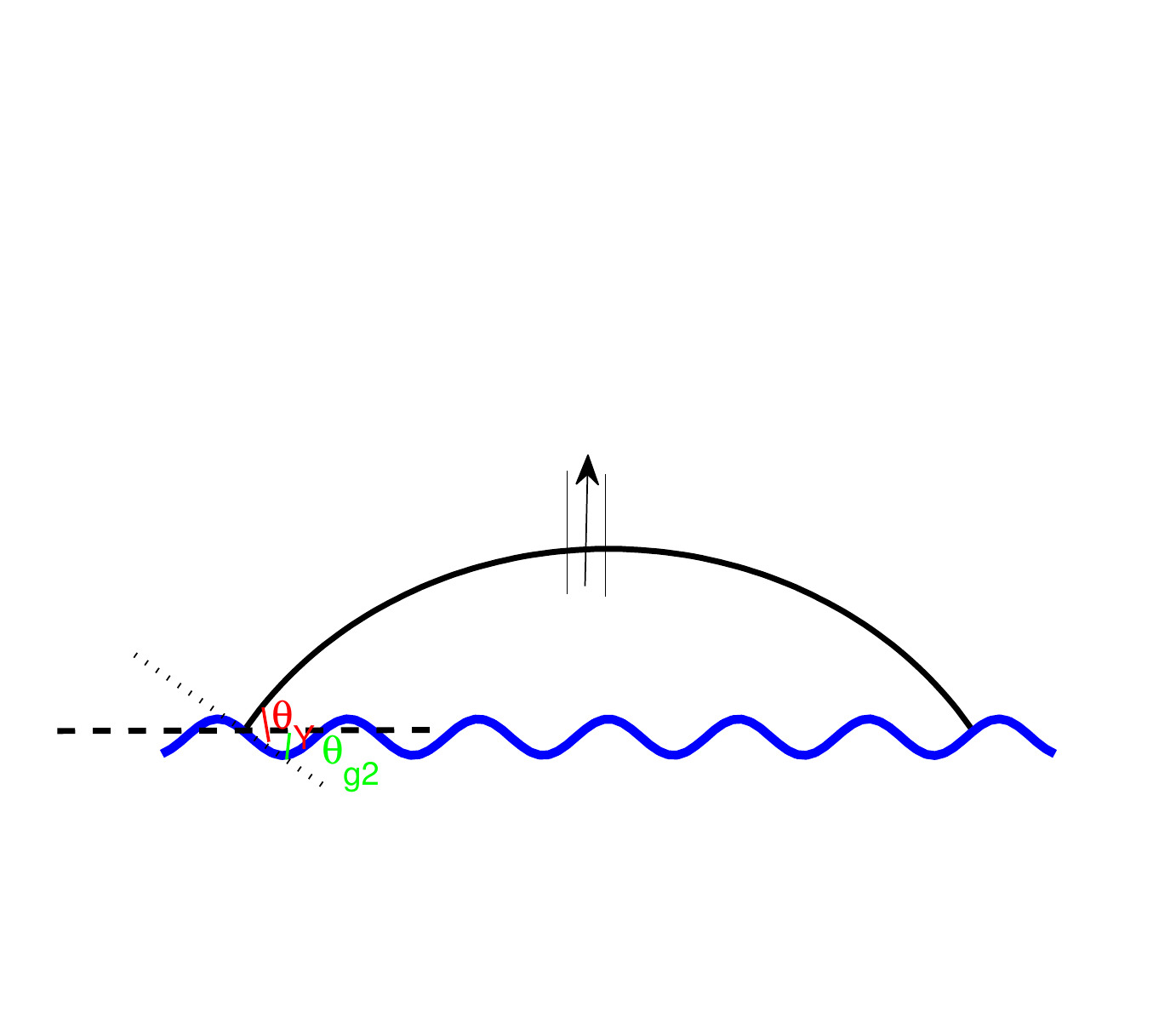}}
    \vspace*{-10mm}
    \caption{The advancing and receding contact angles on a geometrically rough surface.}    \lbl{fig:rough_CAH}
 \end{figure}
 }

We then consider a three-dimensional problem with chemically patterned surface
 as shown in Fig.~\ref{fig:chem_CAL}. The basis is made of a material with Young's angle $\theta_{Y1}$ and the 
patterns correspond to  Young's angle $\theta_{Y2}$ ($\theta_{Y1}>\theta_{Y2}$).
When a liquid drop is put in the center area of the surface,
there might be many contact lines.
Notice the equation \eqref{e:ModifiedWC}  reduces to
\begin{equation*}
\cos \theta_a=\aint_{CL}\cos\theta_Y ds\approx\lambda\cos\theta_{Y1}+(1-\lambda)\cos\theta_{Y2},
\end{equation*}
where $\lambda$  is the length fraction of the contact line in material 1.
This implies that different contact lines may correspond to different apparent angles since
$\lambda$ may change.
Nevertheless, we easily see that the largest apparent 
contact angle is $\theta_{Y1}$ corresponding to a contact line located entirely 
in material 1 (as shown in the left sub-figure of Fig.~\ref{fig:chem_CAL}),
and the smallest one is given by
$$\cos\theta_{rec}=\frac{r}{\eps}\cos\theta_{Y1}+(1-\frac{r}{\eps})\cos\theta_{Y2},$$
corresponding to a contact line periodically crossing the patterns (as shown  in the right sub-figure of Fig.~\ref{fig:chem_CAL}). They
are the advancing and receding contact angles, respectively.
  \begin{figure}[htb!]
\vspace*{-1mm}
    \centering
    {\includegraphics[height=6.cm]{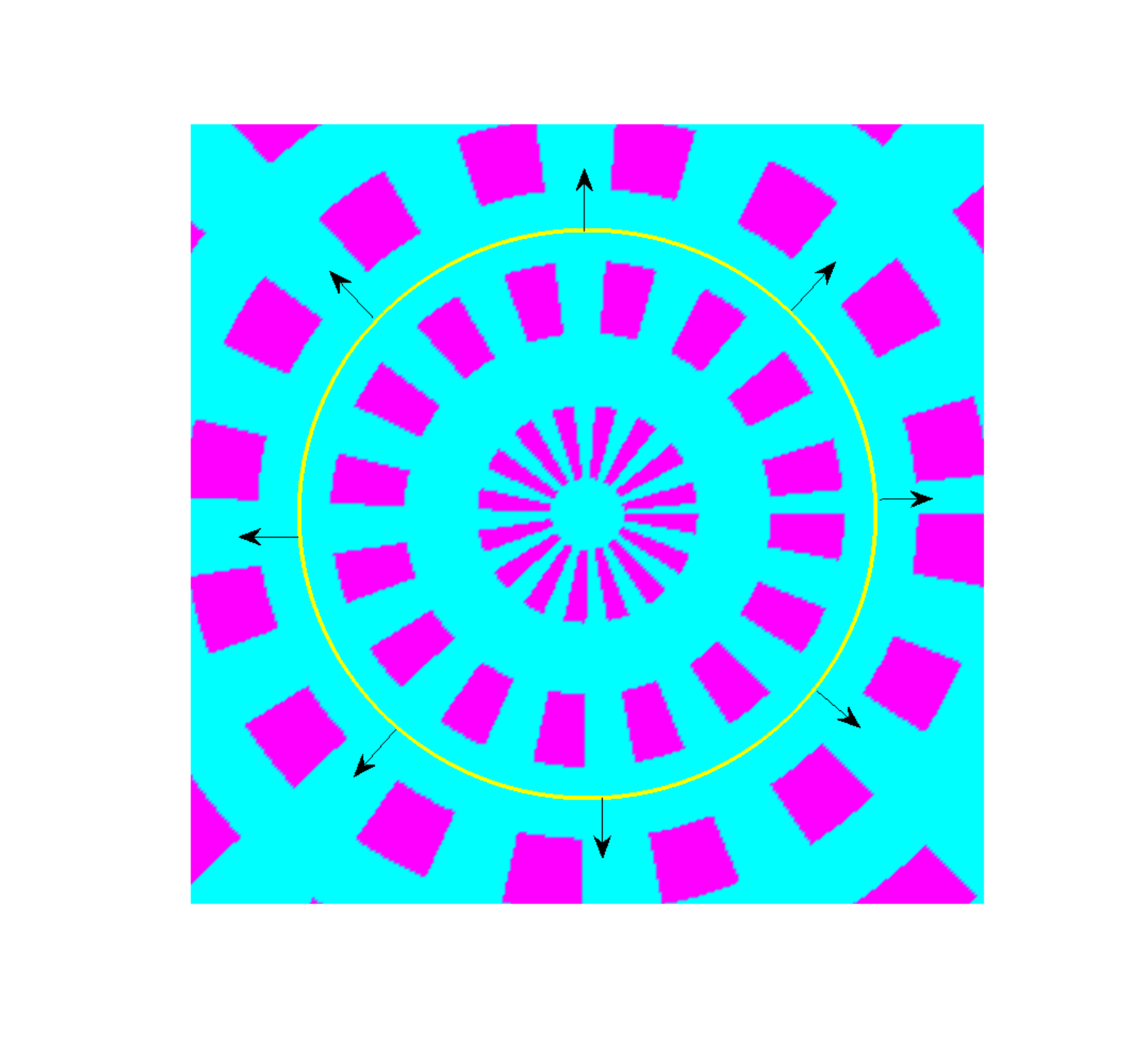}}
    {\includegraphics[height=6.cm]{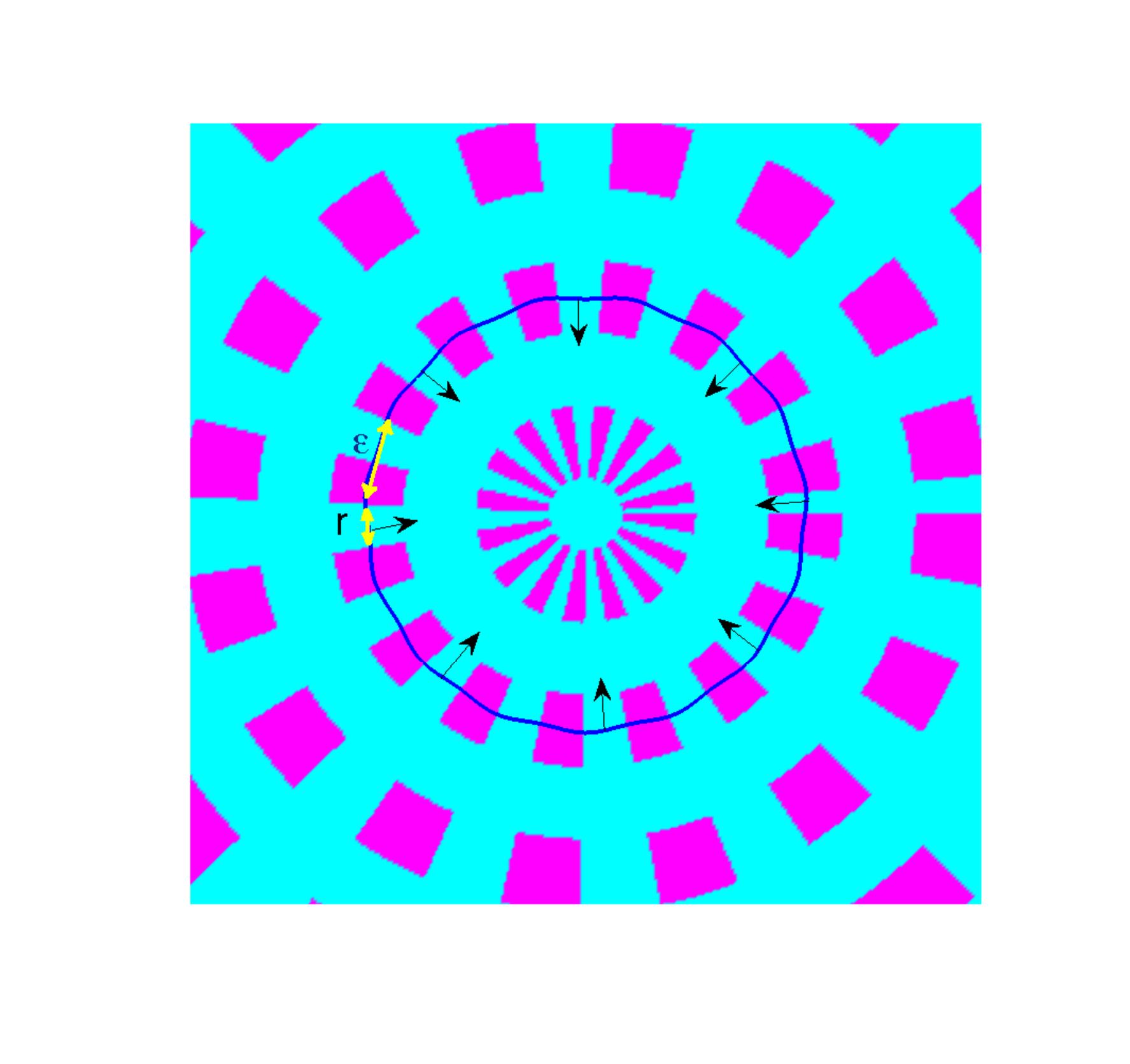}}
    \vspace*{-10mm}
    \caption{The advancing and receding contact lines on a chemically patterned surface.}    \lbl{fig:chem_CAL}
 \end{figure}

For a geometrically rough surface in three dimension, the situation is much more complicated 
than the previous case, since the geometric angle varies on a contact line and the coefficient $\sqrt{1+|\partial_y\phi_\eps|^2}$
 now has an effect. In general,
it is very difficult to compute the apparent contact angle analytically. A numerical computation
might be useful\cite{xuWangWang2016}. However, in some special situation, as shown in Fig.~\ref{fig:rough_CAL},
the solid surface is made rough by periodic pillars with flat tops.
If we assume there is air trapped between the pillars under the liquid,
 the air can be considered as a material with Young's angle $\pi$\cite{Choi09}. 
Then, for the contact line in Fig.~\ref{fig:rough_CAL}, the apparent contact  angle will be given by
$$\cos\theta_{a}=\frac{r}{\eps}\cos\theta_{Y}-(1-\frac{r}{\eps}).$$ 
This equation can be used to characterize the super-hydrophobicity  of
a textured rough surface.
  \begin{figure}[htb!]
\vspace*{0mm}
    \centering
    {\includegraphics[height=3.cm]{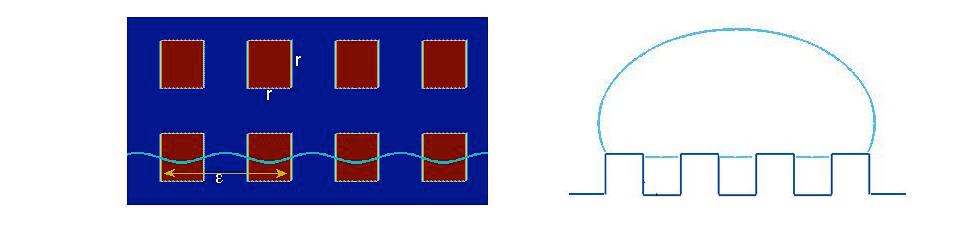}}
    \vspace*{-1mm}
    \caption{A rough surface with periodic pillars: the top view (left) and the side view (right).}    \lbl{fig:rough_CAL}
 \end{figure}

From these examples, we could see that the equation~\eqref{e:ModifiedWC} is quite general
and can be useful in many cases. For example, the equation can be used to understand many experimental results\cite{Raj12,XuWang2013}.
The key information from the equation is that both the geometric
and chemical properties must be averaged on a contact line to give the correct value for
the apparent contact angle. This is not known in literature.
However, there are some restrictions in usage of the formula. Firstly, the equation \eqref{e:ModifiedWC} relies on a  knowledge 
of the position of the contact line, usually which is  not known a priori. In this case, we need some estimations or
 computations for the location of a contact line.  Secondly, the equation is derived from a stationary wetting problem. 
It might need to be adapted for dynamic problems where the apparent contact angle also depends on velocity\cite{guan2016asymmetric,WangXu2016}.

\section{Conclusion}
In this paper, we derive a new formula
for the macroscopic contact angle by a two-scale asymptotic homogenization approach.  The new equation implies that both the geometric
and chemical properties must be averaged on a contact line to give the correct value for
an apparent contact angle. The formula can be reduced to
a modified Wenzel equation for geometrically rough surface and a modified Cassie equation for 
chemically rough surfaces. Unlike the classical Wenzel and Cassie equations, the modified equations
correspond to local minimizers of the energy in the system and can be used to understand the important contact angle hysteresis phenomena.
We prove the homogenization result rigorously by a variational method. The difficulty to prove this result is that the solution of the original problem might not be unique. The key idea of our proof is to construct an auxiliary energy minimizing problem for each solution and to estimate the energy  
directly.  We also discuss how to use the new formula in practices, e.g. when studying the contact angle hysteresis phenomena.

\section*{Acknowledgments}
We thank Professor Xiaoping Wang from the Hong Kong University of Science and Technology for  helpful discussions.

%\begin{thebibliography}{10}
\bibliographystyle{plain}%{unsrt}
\bibliography{literW}

\end{document}